\documentclass[a4paper,11pt]{amsart}

\usepackage[colorlinks,linkcolor=blue,citecolor=blue]{hyperref}
\usepackage{latexsym, amssymb, amsmath, amsthm, mathrsfs, bbm, pstricks, ifthen}

\usepackage[all, knot]{xy}
\xyoption{arc}
\allowdisplaybreaks[4]

\newcommand{\Dchaintwo}[4]{
\rule[-3\unitlength]{0pt}{8\unitlength}
\begin{picture}(14,5)(0,3)
\put(1,2){\ifthenelse{\equal{#1}{l}}{\circle*{2}}{\circle{2}}}
\put(2,2){\line(1,0){10}}
\put(13,2){\ifthenelse{\equal{#1}{r}}{\circle*{2}}{\circle{2}}}
\put(1,5){\makebox[0pt]{\scriptsize #2}}
\put(7,4){\makebox[0pt]{\scriptsize #3}}
\put(13,5){\makebox[0pt]{\scriptsize #4}}
\end{picture}}

\newcommand{\Hexagon}[7]{
\rule[-3\unitlength]{0pt}{8\unitlength}
\begin{picture}(30,20)(0,3)
\put(1,13){\ifthenelse{\equal{#1}{l}}{\circle*{2}}{\circle{2}}}
\put(2,14){\line(1,1){6}}
\put(9,21){\ifthenelse{\equal{#1}{l}}{\circle*{2}}{\circle{2}}}
\put(10,21){\line(1,0){8}}
\put(19,21){\ifthenelse{\equal{#1}{l}}{\circle*{2}}{\circle{2}}}
\put(20,20){\line(1,-1){6}}
\put(27,13){\ifthenelse{\equal{#1}{l}}{\circle*{2}}{\circle{2}}}
\put(26,12){\line(-1,-1){6}}
\put(19,5){\ifthenelse{\equal{#1}{l}}{\circle*{2}}{\circle{2}}}
\put(18,5){\line(-1,0){8}}
\put(9,5){\ifthenelse{\equal{#1}{l}}{\circle*{2}}{\circle{2}}}
\put(8,6){\line(-1,1){6}}
\put(-3,12){\makebox[0pt]{\scriptsize #2}}
\put(9,23){\makebox[0pt]{\scriptsize #3}}
\put(19,23){\makebox[0pt]{\scriptsize #4}}
\put(31,12){\makebox[0pt]{\scriptsize #5}}
\put(19,1){\makebox[0pt]{\scriptsize #6}}
\put(9,1){\makebox[0pt]{\scriptsize #7}}
\end{picture}}

\newcommand{\THexagon}[7]{
\rule[-3\unitlength]{0pt}{8\unitlength}
\begin{picture}(30,20)(0,3)
\put(1,13){\ifthenelse{\equal{#1}{l}}{\circle*{2}}{\circle{2}}}
\put(2,14){\line(1,1){6}}
\put(9,21){\ifthenelse{\equal{#1}{l}}{\circle*{2}}{\circle{2}}}
\put(10,21){\line(1,0){8}}
\put(19,21){\ifthenelse{\equal{#1}{l}}{\circle*{2}}{\circle{2}}}
\put(20,20){\line(1,-1){6}}
\put(27,13){\ifthenelse{\equal{#1}{l}}{\circle*{2}}{\circle{2}}}
\put(26,12){\line(-1,-1){6}}
\put(19,5){\ifthenelse{\equal{#1}{l}}{\circle*{2}}{\circle{2}}}
\put(18,5){\line(-1,0){8}}
\put(9,5){\ifthenelse{\equal{#1}{l}}{\circle*{2}}{\circle{2}}}
\put(8,6){\line(-1,1){6}}
\put(2,13){\line(1,0){24}}
\put(10,21){\line(1,-2){8}}
\put(18,21){\line(-1,-2){8}}
\put(-3,12){\makebox[0pt]{\scriptsize #2}}
\put(9,23){\makebox[0pt]{\scriptsize #3}}
\put(19,23){\makebox[0pt]{\scriptsize #4}}
\put(31,12){\makebox[0pt]{\scriptsize #5}}
\put(19,1){\makebox[0pt]{\scriptsize #6}}
\put(9,1){\makebox[0pt]{\scriptsize #7}}
\end{picture}}

\newcommand{\Tchainfour}[5]{
\rule[-4\unitlength]{0pt}{5\unitlength}
\begin{picture}(20,20)(0,3)
\put(2,4){\circle{2}}
\put(3,4){\line(1,0){10}}
\put(3,5){\line(1,1){10.2}}
\put(13,15){\line(-1,-1){10.2}}
\put(2,5){\line(0,1){10}}
\put(14,4){\circle{2}}
\put(14,5){\line(0,1){10}}
\put(13,5){\line(-1,1){10.2}}
\put(3,15){\line(1,-1){10.2}}
\put(14,16){\circle{2}}
\put(2,16){\circle{2}}
\put(2,3){\makebox[0pt]{\scriptsize #1}}
\put(2,1){\makebox[0pt]{\scriptsize #2}}
\put(14,1){\makebox[0pt]{\scriptsize #3}}
\put(2,18){\makebox[0pt]{\scriptsize #4}}
\put(14,18){\makebox[0pt]{\scriptsize #5}}
\end{picture}}

\newcommand{\Dchainfour}[8]{
\rule[-3\unitlength]{0pt}{5\unitlength}
\begin{picture}(38,5)(0,3)
\put(1,2){\ifthenelse{\equal{#1}{1}}{\circle*{2}}{\circle{2}}}
\put(2,2){\line(1,0){10}}
\put(13,2){\ifthenelse{\equal{#1}{2}}{\circle*{2}}{\circle{2}}}
\put(14,2){\line(1,0){10}}
\put(25,2){\ifthenelse{\equal{#1}{3}}{\circle*{2}}{\circle{2}}}
\put(26,2){\line(1,0){10}}
\put(37,2){\ifthenelse{\equal{#1}{4}}{\circle*{2}}{\circle{2}}}
\put(1,5){\makebox[0pt]{\scriptsize #2}}
\put(7,4){\makebox[0pt]{\scriptsize #3}}
\put(13,5){\makebox[0pt]{\scriptsize #4}}
\put(19,4){\makebox[0pt]{\scriptsize #5}}
\put(25,5){\makebox[0pt]{\scriptsize #6}}
\put(31,4){\makebox[0pt]{\scriptsize #7}}
\put(37,5){\makebox[0pt]{\scriptsize #8}}
\end{picture}}

\newcommand{\Fchainfour}[5]{
\rule[-4\unitlength]{0pt}{5\unitlength}
\begin{picture}(20,20)(0,3)
\put(2,4){\circle{2}}
\put(3,4){\line(1,0){10}}
\put(3,5){\line(1,1){10.2}}
\put(13,15){\line(-1,-1){10.2}}
\put(2,5){\line(0,1){10}}
\put(14,4){\circle{2}}
\put(14,5){\line(0,1){10}}
\put(13,5){\line(-1,1){10.2}}
\put(3,15){\line(1,-1){10.2}}
\put(14,16){\circle{2}}
\put(13,16){\line(-1,0){10}}
\put(2,16){\circle{2}}
\put(2,3){\makebox[0pt]{\scriptsize #1}}
\put(2,1){\makebox[0pt]{\scriptsize #2}}
\put(14,1){\makebox[0pt]{\scriptsize #3}}
\put(2,18){\makebox[0pt]{\scriptsize #4}}
\put(14,18){\makebox[0pt]{\scriptsize #5}}
\end{picture}}

\newcommand{\Echainfour}[5]{
\rule[-4\unitlength]{0pt}{5\unitlength}
\begin{picture}(20,20)(0,3)
\put(2,4){\circle{2}}
\put(3,4){\line(1,0){10}}
\put(2,5){\line(0,1){10}}
\put(14,4){\circle{2}}
\put(14,5){\line(0,1){10}}
\put(14,16){\circle{2}}
\put(13,16){\line(-1,0){10}}
\put(2,16){\circle{2}}
\put(2,3){\makebox[0pt]{\scriptsize #1}}
\put(2,1){\makebox[0pt]{\scriptsize #2}}
\put(14,1){\makebox[0pt]{\scriptsize #3}}
\put(2,18){\makebox[0pt]{\scriptsize #4}}
\put(14,18){\makebox[0pt]{\scriptsize #5}}
\end{picture}}

\def \To{\longrightarrow}
\def \dim{\operatorname{dim}}
\def \deg{\operatorname{deg}}
\def \D{\operatorname{d}}
\def \gr{\operatorname{gr}}

\def \ad{\operatorname{ad}}
\def \Vec{\operatorname{Vec}}
\def \comod{\operatorname{comod}}

\def \N{\mathbb{N}}

\def \D{\Delta}

\def \e{\varepsilon}

\def \Z{\mathbb{Z}}

\def \k{\mathbbm{k}}
\def \1{\mathbf{1}}
\def \id{\operatorname{id}}

\numberwithin{equation}{section}

\newtheorem{theorem}{Theorem}[section]
\newtheorem{lemma}[theorem]{Lemma}
\newtheorem{proposition}[theorem]{Proposition}
\newtheorem{corollary}[theorem]{Corollary}

\newtheorem{definition}[theorem]{Definition}
\newtheorem{example}[theorem]{Example}
\newtheorem{remark}[theorem]{Remark}
\newtheorem{conjecture}[theorem]{Conjeture}

\begin{document}

\title[Finite Quasi-Quantum Groups over Abelian Groups]{On the Classification of Finite Quasi-Quantum Groups over Abelian Groups}

\author{Hua-Lin Huang, Gongxiang Liu, Yuping Yang*, and Yu Ye}
\thanks{*Corresponding author}

\address{School of Mathematical Sciences, Huaqiao University, Quanzhou 362021, China} \email{hualin.huang@hqu.edu.cn}

\address{Department of Mathematics, Nanjing University, Nanjing 210093, China}
\email{gxliu@nju.edu.cn}

\address{School of Mathematics and statistics, Southwest University, Chongqing 400715, China}
\email{yupingyang@swu.edu.cn}

\address{\begin{minipage}{14.5cm}{School of Mathematical Sciences, Wu Wen-Tsun Key Laboratory of Mathematics, University of Science and Technology of China, Hefei, Anhui 230026, CHINA 
			\\
			and 
			\\
			Hefei National Laboratory, University of Science and Technology of China, Hefei 230088, CHINA}\end{minipage}}
\email{yeyu@ustc.edu.cn}

\subjclass[2010]{19A22, 18D10, 16G20}
\keywords{quasi-Hopf algebra,  quantum group, pointed tensor category}
\date{}
\maketitle

\begin{abstract} 
Using a variety of methods developed in the theory of finite-dimensional quasi-Hopf algebras, we classify all finite-dimensional coradically graded  pointed coquasi-Hopf algebras over abelian groups.  As a consequence, we partially confirm the generation conjecture of pointed finite tensor categories due to Etingof, Gelaki, Nikshych and Ostrik.
\end{abstract}

\section{Introduction}
As a continuation to a series of previous works \cite{HLYY, HLYY2, HYZ}, this paper completes the classification problem of finite-dimensional coradically graded  pointed coquasi-Hopf algebras over abelian groups. Throughout, we work over an algebraically closed field $\k$ of characteristic zero. Unless stated otherwise, in this paper all spaces, maps, (co)algebras, (co)modules, and categories, etc., are over $\k.$

The classification of finite-dimensional pointed Hopf algebras over finite abelian groups was completed over the last two decades, and a systematic approach (in particular Weyl groupoids and arithmetic root systems) was established, see \cite{AS2,A2,A3, H0,H4}.  Meanwhile, Etingof and Gelaki proposed to classify pointed finite tensor categories. By the Tannakian formalism \cite{EGNO}, this amounts to a classification of certain finite quasi-quantum groups, namely finite-dimensional elementary quasi-Hopf algebras, or dually finite-dimensional pointed coquasi-Hopf algebras. In the pioneering works \cite{EG1,EG2, EG3,G, A}, a few examples and classification results of such algebras, and consequently the associated pointed finite tensor categories, are thus obtained.
In \cite{HLYY, HLYY2, HYZ}, the authors of the present paper continue the study of the classification problem of finite quasi-quantum groups and several interesting classification results are also obtained. To explain this and our main result of this paper, we need some concrete notations. 

Once and for all, let $G$ be a finite abelian group and $\Phi$ be a $3$-cocycle on $G.$ Based on an analog of the lifting method in the theory of finite-dimensional pointed Hopf algebras, a complete understanding of the Nichols algebras in the twisted Yetter-Drinfeld module category $_{\k G}^{\k G}\mathcal{Y}\mathcal{D}^{\Phi}$ is the crux for the classification of finite-dimensional pointed coquasi-Hopf algebras. A twisted Yetter-Drinfeld module $V \in {_{\k G}^{\k G}\mathcal{Y}\mathcal{D}^{\Phi}}$ is said to be of diagonal type, if it is a direct sum of $1$-dimensional twisted Yetter-Drinfeld modules. The associated Nichols algebra $B(V)$ is called diagonal if $V$ is so. Let $H$ be a pointed coquasi-Hopf algebra over $G$,  and $\gr(H)$ the coradically graded coquasi-Hopf algebra associated to $H$. The coinvariant subalgebra $R$ of $\gr(H)$ will be a twisted Yetter-Drinfeld module in $_{\k G}^{\k G}\mathcal{Y}\mathcal{D}^{\Phi}$ for certain $\Phi$, and $H$ is called diagonal if $R$ is diagonal as a twisted Yetter-Drinfeld module.
 In  \cite{HLYY, HLYY2}, we classified all finite-dimensional Nichols algebras of diagonal type in  ${_{\k G}^{\k G}\mathcal{Y}\mathcal{D}^{\Phi}}$ and proved that every finite-dimensional pointed coquasi-Hopf algebra of diagonal type must be the form of $B(V)\#\k G$. 
 
The aim of this paper is to study Nichols algebras of nondiagonal type and the main result is the following (see Theorem \ref{t3.1} for an equivalent form) .

{\bf Theorem 0.1}.
Let $B(V)\in\ _{\k G}^{\k G} \mathcal{YD}^\Phi$ be a Nichols algebra of nondiagonal type with  $G_V=G$. Then $B(V)$ is infinite dimensional.

Here $G_V$ is the support group of $V$ (see the paragraph after Definition \ref{d2.9}). Under assumption that $G_V=G$, which is natural for us since the braided Hopf algebra structure of $B(V)$ is determined by $G_V$ rather than $G$ (see the paragraph after Proposition \ref{p2.6}), our result tells us that every finite-dimensional Nichols algebra $B(V)\in\ _{\k G}^{\k G} \mathcal{YD}^\Phi$ must be of diagonal type finally, which already was classified in our previous works. As consequences, we can get the structure of general finite-dimensional pointed coquasi-Hopf algebras now (see Theorems \ref{t9.1} and \ref{t5.2}).

{\bf Theorem 0.2}. If $M$ is a  finite-dimensional  pointed coquasi-Hopf algebra over finite abelian group $G$, then $\gr(M)\cong B(V)\# \k G$ for a Yetter-Drinfeld module of finite type $V\in\ {_{\k G}^{\k G} \mathcal{YD}^\Phi}$.

In \cite{EGNO}, Etingof, Gelaki, Nikshych and Ostrik conjecture that every pointed finite tensor category over a field of characteristic zero is tensor generated by objects of length $2$. Let $G(\mathcal{C})$ be the set of isomorphism classes of simple objects in a pointed tensor category $\mathcal{C}$. Then $G(\mathcal{C})$ is naturally a group under tensor product. Using our classification result and some useful result in \cite{HLYY2}, we can partially prove the conjecture (see Theorem \ref{t5.5}).

{\bf Theorem 0.3}. Let $\mathcal{C}$ be a pointed finite tensor category over a field of characteristic zero. If the group $G(\mathcal{C})$ is abelian, then $\mathcal{C}$ is tensor generated by objects of length $2$. 

The paper is organized as follows. In Section 2, we recall some necessary notions and particularly introduce a method to study Nichols algebras in  ${_{\k G}^{\k G} \mathcal{YD}^\Phi}$ called change of based groups. Sections 3 and 4 are designed to give a proof of the above Theorem 0.1. The last section is devoted to the classification of  finite-dimensional pointed coquasi-Hopf algebras and the generation problem of pointed finite tensor categories.

\section{Preliminaries}
In this section, we recall some necessary notions and basic facts about pointed coquasi-Hopf algebras, twisted Yetter-Drinfeld modules and Nichols algebras. The reader is referred to \cite{EGNO, HLYY, HLYY2} for related concepts and notations.

\subsection{Pointed coquasi-Hopf algebras} 
A coquasi-Hopf algebra is a coalgebra $(H,\D,\e)$ equipped with a
compatible quasi-algebra structure and a quasi-antipode. Namely,
there exist two coalgebra homomorphisms $$m\colon H \otimes H \To H, \ a
\otimes b \mapsto ab \quad \text{and} \quad \mu\colon \k \To H,\ \lambda \mapsto \lambda
1_H,$$ a convolution-invertible map $\Phi\colon H^{\otimes 3} \To \k$
called the \emph{associator}, a coalgebra antimorphism $S\colon H \To H$ and two
functions $\alpha,\beta\colon H \To \k$ such that for all $a,b,c,d \in
H$ the following equalities hold:
\begin{eqnarray*}
&a_1(b_1c_1)\Phi(a_2,b_2,c_2)=\Phi(a_1,b_1,c_1)(a_2b_2)c_2,\\
&1_H a=a=a1_H, \\
&\Phi(a_1,b_1,c_1d_1)\Phi(a_2b_2,c_2,d_2) =\Phi(b_1,c_1,d_1)\Phi(a_1,b_2c_2,d_2)\Phi(a_2,b_3,c_3),\\
&\Phi(a,1_H,b)=\e(a)\e(b). \\
&S(a_1)\alpha(a_2)a_3=\alpha(a)1_H, \quad a_1\beta(a_2)S(a_3)=\beta(a)1_H, \label{2.11} \\
&\Phi(a_1,S(a_3),a_5)\beta(a_2)\alpha(a_4)
=\Phi^{-1}(S(a_1),a_3, S(a_5)) \alpha(a_2)\beta(a_4)=\e(a).
 \end{eqnarray*}
The triple $(S,\alpha,\beta)$ is called a quasi-antipode. $H$ is called a {\bf pointed coquasi-Hopf algebra} if $(H,\D,\e)$ is a pointed coalgebra, i.e., every simple comodule of $H$ is $1$-dimensional.

Let $C$ be a coalgebra, the coradical $C_0$ of $C$ is the sum of all simple subcoalgebras of $C$. Fix a coalgebra $C$ with coradical $C_0$, define $C_n$ inductively as follows: for each $n\geq 1$, define $$C_n=\D^{-1}(C\otimes C_{n-1}+C_0\otimes C).$$
Then we get a filtration $C_0\subset C_1\subset \cdots C_n\subset \cdots, $ which is called the \emph{coradical filtration} of $C$. A coquasi-Hopf algebra has a coradical filtration since it is a coalgebra.

Given a coquasi-Hopf algebra $(H,\D, \e,m, \mu, \Phi,S,\alpha,\beta),$ let $\{H_n\}_{n \ge 0}$ be its coradical filtration, and let $$\gr H = H_0 \oplus H_1/H_0 \oplus H_2/H_1
\oplus \cdots,$$ the corresponding coradically graded coalgebra. Then naturally
$\gr H$ inherits from $H$ a graded coquasi-Hopf algebra structure. The
corresponding graded associator $\gr\Phi$ satisfies
$\gr\Phi(\bar{a},\bar{b},\bar{c})=0$ for all homogeneous elements
$\bar{a},\bar{b},\bar{c} \in \gr H$ unless they all lie in $H_0.$
Similar conditions hold for $\gr\alpha$ and $\gr\beta.$ A coquasi-Hopf algebra $H$ is called {\bf coradically graded} if $H\cong \gr(H)$ as coquasi-Hopf algebras.

If $H$ is a pointed coquasi-Hopf algebra, then $H_0$ is a pointed cosemisimple coquasi-Hopf algebra, which is determined by a group $G$ together with a $3$-cocycle on $G$ as follows. 

\begin{example}
Let $G$ be a group. Clearly the group algebra $\k G$ is a Hopf algebra with $\D(g)=g\otimes g,\; S(g)=g^{-1}$ and
$\e(g)=1$ for any $g\in G$. Let $\omega$ be a normalized $3$-cocycle on $G$, i.e.
\begin{eqnarray}
&\omega(ef,g,h)\omega(e,f,gh)=\omega(e,f,g)\omega(e,fg,h)\omega(f,g,h),\\
&\omega(f,1,g)=1
\end{eqnarray}
for all $e,f,g,h \in G$. By linearly extending, $\omega \colon (\k G)^{\otimes 3} \to \k$ becomes a
 convolution-invertible map. Define two linear functions
 $\alpha,\beta \colon \k G \to \k$ by \[ \alpha(g):=\e(g) \quad \text{and} \quad \beta(g):=\frac{1}{\omega(g,g^{-1},g)} \]
 for any $g\in G$. Then $kG$ together with these $\omega, \ \alpha$ and $\beta$
makes a coquasi-Hopf algebra, which will be written as $(\k G,\omega)$ in the following. The comodule category of $(\k G,\omega)$ forms a tensor category, which is called a Gr-category and denoted by $\Vec_G^\omega$.
\end{example}

Let us now consider the construction of Gr-categories which will be especially important in this paper. The crux to determine all the Gr-categories is to give a complete list of the representatives of the $3$-cohomology classes in $\operatorname{H}^3(G,\k^*)$ for all groups $G.$ However, when $G$ is a finite abelian group, the problem was solved in \cite{HLYY2}, and a list of the representatives of $\operatorname{H}^3(G,\k^*)$ can be given as follows.

Let $\N$ denote the set of nonnegative integers, $\Z$ the ring of integers, and $\Z_m$ the cyclic group of order $m.$ Any finite abelian group $G$ is of the form $\mathbb{Z}_{m_{1}}\times\cdots \times\mathbb{Z}_{m_{n}}$ with $m_j\in \mathbb{N}$
for $1\leq j\leq n.$  Denote by $\mathcal{A}$ the set of all $\N$-sequences
\begin{equation}\label{cs}(c_{1},\ldots,c_{l},\ldots,c_{n},c_{12},\ldots,c_{ij},\ldots,c_{n-1,n},c_{123},
\ldots,c_{rst},\ldots,c_{n-2,n-1,n})\end{equation}
such that $ 0\leq c_{l}<m_{l}, \ 0\leq c_{ij}<(m_{i},m_{j}), \ 0\leq c_{rst}<(m_{r},m_{s},m_{t})$ for $1\leq l\leq n, \ 1\leq i<j\leq n, \ 1\leq r<s<t\leq n$, where $c_{ij}$ and $c_{rst}$ are ordered in the lexicographic order of their indices. We denote by $\underline{\mathbf{c}}$ the sequence \eqref{cs}  in the following. Let $g_i$ be a generator of $\mathbb{Z}_{m_{i}}, 1\leq i\leq n$. For any $\underline{\mathbf{c}}\in \mathcal{A}$, define
\begin{eqnarray}\label{eq2.4}
&& \omega_{\underline{\mathbf{c}}}\colon G\times G\times G\To \k^{\ast} \notag \\
&&[g_{1}^{i_{1}}\cdots g_{n}^{i_{n}},g_{1}^{j_{1}}\cdots g_{n}^{j_{n}},g_{1}^{k_{1}}\cdots g_{n}^{k_{n}}] \mapsto \\ && \prod_{l=1}^{n}\zeta_{m_l}^{c_{l}i_{l}[\frac{j_{l}+k_{l}}{m_{l}}]}
\prod_{1\leq s<t\leq n}\zeta_{m_{t}}^{c_{st}i_{t}[\frac{j_{s}+k_{s}}{m_{s}}]}
\prod_{1\leq r<s<t\leq n}\zeta_{(m_{r},m_{s},m_{t})}^{c_{rst}i_{r}j_{s}k_{t}}. \notag
\end{eqnarray}
Here and below $\zeta_m$ stands for an $m$-th primitive root of unity.
According to \cite[Proposition 3.8]{HLYY2},   $\{\omega_{\underline{\mathbf{c}}} \mid \underline{\mathbf{c}}\in \mathcal{A}\}$ forms a complete set of representatives of the normalized $3$-cocycles on $G$ up to $3$-cohomology.

\subsection{Nichols algebras of Yetter-Drinfeld modules}
Nichols algebras are very important for the construction of pointed coquasi-Hopf algebras. For our purpose, we are mainly concerned with the Nichols algebras in the Yetter-Drinfeld module category of the coquasi-Hopf algebra $(\k G,\Phi)$, where $G$ is a finite abelian group and $\Phi$ is a normalized 3-cocycle on $G$. To emphasize $\Phi,$ we denote the Yetter-Drinfeld category of $(\k G,\Phi)$ as $_{\k G}^{\k G}\mathcal{Y}\mathcal{D}^{\Phi}$. For convenience, we call an object in  $_{\k G}^{\k G}\mathcal{Y}\mathcal{D}^{\Phi}$ a twisted Yetter-Drinfeld module.  Define
\begin{equation}\label{e2.5}
\widetilde{\Phi}_g(x,y)=\frac{\Phi(g,x,y)\Phi(x,y,g)}{\Phi(x,g,y)}
\end{equation}
 for all $g,x,y\in G$. By direct computation one can show that $\widetilde{\Phi}_g$ is a 2-cocycle on $G$. The construction of category of twisted Yetter-Drinfeld modules can be summarized as follows, the detailed computations can be found in \cite{HLYY2,HY}.

\begin{definition}
An object in $_{\k G}^{\k G}\mathcal{Y}\mathcal{D}^{\Phi}$ is a  $G$-graded vector space $V=\oplus_{g\in G}V_g$ \emph{ (}$V_g=\{v\in V|\rho(v)=g\otimes v\}$ as a $kG$-comodule\emph{)} with each $V_g$ a projective
$G$-representation with respect to the 2-cocycle $\widetilde{\Phi}_g,$ namely for any $e,f\in G, v\in V_g$ we have 
\begin{equation}\label{eq2.6}
e\triangleright(f\triangleright v)=\widetilde{\Phi}_g(e,f) (ef)\triangleright v.
\end{equation}
The module structure on the tensor product $V_g\otimes V_h$ is determined by
\begin{equation}\label{n3.8}
e\triangleright (X\otimes Y)=\widetilde{\Phi}_e(g,h)(e\triangleright X)\otimes (e\triangleright Y), \ X\in V_g, \ Y\in V_h.
\end{equation}
The associativity and the braiding constraints of $_{\k G}^{\k G}\mathcal{Y}\mathcal{D}^{\Phi}$ are given respectively by
\begin{eqnarray}
&\ a_{V_e,V_f,V_g}((X\otimes Y)\otimes Z) =\Phi(e,f,g)^{-1} X\otimes (Y\otimes Z )\\
&R(X\otimes Y)=e \triangleright Y\otimes X
\end{eqnarray}
for all $X\in V_e,\  Y\in V_f,\  Z \in V_g.$
\end{definition}

 Let $\Phi$ be a $3$-cocycle on $G$ as given in \eqref{eq2.4}.
One can verify directly that
\begin{equation}
\widetilde{\Phi}_g\widetilde{\Phi}_h=\widetilde{\Phi}_{gh},\  \forall g,h\in G.
\end{equation}
Suppose $V_g$ is $(G,\widetilde{\Phi}_g)$-representation, $V_h$ is a $(G,\widetilde{\Phi}_h)$-representation, then $V_g\otimes V_h$ is a $(G,\widetilde{\Phi}_{gh})$-representation.
In particular, the dual object $V_g^*$ of $V_g$ is a $(G,\widetilde{\Phi}_{g^{-1}})$-representation and $(V_g^*)^*=V_g$, see \cite[Proposition 2.5]{HYZ} for details. 

A Yetter-Drinfeld module $V\in {^{\k G}_{\k G}\mathcal{Y}\mathcal{D}^\Phi}$ is called diagonal if $V$ is direct sum of $1$-dimensional twisted Yetter-Drinfeld modules. For a simple Yetter-Drinfeld module $V$ in ${_{\k G}^{\k G}\mathcal{Y}\mathcal{D}^{\Phi}}$, there exists some $g \in G$ such that $V=V_g$ and we define $g_V:=g$ in this case. Recall that a $2$-cocycle $\varphi$ on $G$ is called symmetric if $\varphi(g,h)=\varphi(h,g)$ for all $h,g\in G$. By \eqref{eq2.6},  it is not hard to show that a simple Yetter-Drinfeld module $V$ with $g_V=g$ is $1$-dimensional if and only if $\widetilde{\Phi}_g$ is symmetric. 

Let $V$ be a nonzero object in $^{\k G}_{\k G}\mathcal{Y}\mathcal{D}^\Phi.$ By $T(V)$ we denote the tensor algebra in $^{\k G}_{ \k G}\mathcal{Y}\mathcal{D}^\Phi$ generated freely by $V.$ It is clear that $T(V)$ is isomorphic to $\bigoplus_{n \geq 0}V^{\otimes \overrightarrow{n}}$ as a linear space, where $V^{\otimes \overrightarrow{n}}$ means
$\underbrace{(\cdots((}_{n-1}V\otimes V)\otimes V)\cdots \otimes V).$ This induces a natural $\mathbb{N}$-graded structure on $T(V).$ Define a comultiplication on $T(V)$ by $\Delta(X)=X\otimes 1+1\otimes X, \ \forall X \in V,$ a counit by $\varepsilon(X)=0,$ and an antipode by $S(X)=-X.$ These provide a graded Hopf algebra structure on $T(V)$ in the braided tensor category $^{\k G}_{\k G}\mathcal{Y}\mathcal{D}^\Phi.$

\begin{definition}\label{d2.9}
The Nichols algebra $B(V)$ of $V$ is the quotient Hopf algebra $T(V)/I$ in $^{\k G}_{\k G}\mathcal{Y}\mathcal{D}^\Phi,$ where $I$ is the unique maximal graded Hopf ideal contained in $\bigoplus_{n \geq 2}V^{\otimes \overrightarrow{n}}$.
\end{definition}

A Nichols algebra $B(V)$ is called {\bf of diagonal type} if $V$ is diagonal. Suppose $V=\bigoplus_{i=1}^n V_i\in {^{\k G}_{\k G}\mathcal{Y}\mathcal{D}^\Phi}$ is direct sum of simple objects, then we will say that the rank of $B(V)$ is $n$. According to \cite[Proposition 3.1]{HYZ}, $B(V)$ is a $\mathbb{Z}^n$-graded algebra with  $\deg V_i=e_i$, where $\{e_i:1\leq i\leq n\}$ is a set of free generators of $\mathbb{Z}^n$. 
For $V\in {^{\k G}_{\k G}\mathcal{Y}\mathcal{D}^\Phi}$, we will call $G$ the {\bf based group} of $V$ and $B(V)$. Let $V=V_1\oplus V_2\oplus \cdots \oplus V_n$ be direct sum of simple Yetter-Drinfeld modules in ${^{\k G}_{\k G}\mathcal{Y}\mathcal{D}^\Phi}$, $g_i$ the degree of $V_i$ for $1\leq i\leq n$, the subgroup $G_V:=\langle g_1,g_2,\cdots g_n\rangle$ will be called the {\bf support group} of $V$. 

Next we will recall the definition of the twisting of a Nichols algebra through a 2-cochain of $G$. Let $(V,\triangleright, \delta_{L})\in {^{\k G}_{\k G}\mathcal{YD}^\Phi}$,  and let $J$ be a $2$-cochain of $G$. There is a new action $\triangleright_J$ of $G$ on $V$ determined by
\begin{equation}\label{3.4}
g\triangleright_J X=\frac{J(g,x)}{J(x,g)}g\triangleright X
\end{equation}
 for homogenous element $X\in V$ and $g\in G.$ Here $x=\deg(X)$ is the $G$-degree of $X$.  We denote $(V, \triangleright_{J},\delta_L)$ by $V^{J}$, and one can verify that $V^J\in {^{\k G}_{\k G}\mathcal{Y}\mathcal{D}^{\Phi\ast \partial(J)}}.$ Moreover there is a tensor equivalence $$(F_J,\varphi_0,\varphi_2)\colon\ ^{\k G}_{\k G}\mathcal{Y}\mathcal{D}^\Phi \to {^{\k G}_{\k G}\mathcal{Y}\mathcal{D}^{\Phi\ast\partial (J)}},$$ which takes $V$ to $V^J$ and $$\varphi_2(U,V)\colon\ (U\otimes V)^J\to U^J\otimes V^J,\ \ Y\otimes Z\mapsto J(y,z)^{-1}Y\otimes Z$$ for $Y\in U,\  Z\in V.$

Let $B(V)$ be a Nichols algebra in $^{\k G}_{\k G}\mathcal{Y}\mathcal{D}^\Phi.$ Then $B(V)^J$ is a Hopf algebra in $^{\k G}_{\k G}\mathcal{Y}\mathcal{D}^{\Phi\ast \partial J}$ with multiplication $\circ$ determined by
\begin{equation}\label{eq2.14}
X\circ Y=J(x,y)XY
\end{equation} for all homogenous elements $X,Y\in B(V),$  $x=\deg X, \ y=\deg Y$. Using the same terminology as for coquasi-Hopf algebras, we say that $B(V)$ and $B(V)^{J}$ are twist equivalent, or $B(V)^{J}$ is a twisting of $B(V)^{J}$.
\begin{lemma}\cite[Lemma 2.12]{HLYY2}\label{l2.6}
The twisting $B(V)^J$ of $B(V)$ is a Nichols algebra in $^{\k G}_{\k G}\mathcal{Y}\mathcal{D}^{\Phi\ast \partial J}$ and $B(V)^J\cong B(V^J)$.
\end{lemma}

\subsection{Reduction}

The study of Nichols algebras in $^{\k G}_{\k G}\mathcal{Y}\mathcal{D}^\Phi$ is deeply related on the $3$-cocycle $\Phi$ on $G$. Recall that a $3$-cocycle $\Phi$ on $G$ is called an {\bf abelian $3$-cocycle} if $^{\k G}_{\k G}\mathcal{Y}\mathcal{D}^\Phi$ is pointed, i.e. each simple object of $^{\k G}_{\k G}\mathcal{Y}\mathcal{D}^\Phi$ is $1$-dimensional.  
A key observation in \cite{HLYY2} is that every Nichols algebra in $^{\k G}_{\k G}\mathcal{Y}\mathcal{D}^\Phi$ is twist equivalent to a Nichols algebra in a normal Yetter-Drinfeld category when $\Phi$ is an abelian $3$-cocycle.
Suppose $G=\Z_{m_1}\times \Z_{m_2}\times \cdots\times \Z_{m_n}$, $e_i$ is a generator of $\Z_{m_i}$ for all $1\leq i\leq n$, and $\Phi$ is an abelian $3$-cocycle on $G.$  Then up to cohomology $\Phi$ must be of the form
\begin{equation}\label{eq2.13}
\Phi(e_1^{i_1}\cdots e_n^{i_n},e_1^{j_1}\cdots e_n^{j_n},e_1^{k_1}\cdots e_N^{k_n})
=\prod_{l=1}^n\zeta_l^{c_li_l[\frac{j_l+k_l}{m_l}]}\prod_{1\leq s<t\leq n}\zeta_{m_t}^{c_{st}i_t[\frac{j_s+k_s}{m_s}]}.
\end{equation}
The folowing lemma follows immediately.
\begin{lemma}\label{l2.5}
\begin{itemize}
\item[(1).] Each $3$-cocycle of a finite cyclic group or direct sum of two finite cyclic groups is abelian;
\item[(2).] Let $B(V)\in {^{\k G}_{\k G}\mathcal{Y}\mathcal{D}^\Phi}$ be a Nichols algebra of rank $1$ or rank $2$, then $\Phi_{G_V}$ is an abelian $3$-cocycle on $G_V$.
\end{itemize}
\end{lemma}
Suppose $G=\mathbb{Z}_{m_1}\times\cdots\times \mathbb{Z}_{m_n} = \langle g_1 \rangle \times \cdots \langle g_n\rangle $. Associated to $G$ there is a finite group $\widehat{G}$ defined by
\begin{equation}\label{e4.2}
\widehat{G}=\mathbb{Z}_{m_1^2}\times\cdots\times \mathbb{Z}_{m_n^2} = \langle h_1 \rangle \times \cdots \langle h_n \rangle 
\end{equation}  
Let
\begin{equation}\label{e4.3}
\pi\colon \widehat{G}\to G,\;\;\;\;h_{i}\mapsto g_{i},\;\;\;\;\;1\le i\le n
\end{equation}
be the canonical epimorphism. The following proposition is very important for the study of Nichol algebras of diagonal type.

\begin{proposition}\cite[Proposition 3.15]{HLYY2}\label{p2.6}
Suppose that $\Phi$ is an abelian $3$-cocycle on $G.$ Then $\pi^*\Phi$ is a $3$-coboundary on $\widehat{G}$, namely, there is a $2$-cochain $J$ of $\widehat{G}$ such that $\partial J=\pi^*\Phi$.
\end{proposition}

Since the Yetter-Drinfeld module structure of $B(V)$ is determined by the based group $G$, while the braided 
Hopf algebra structure of $B(V)$ and the braiding are determined by the support group $G_V$. So if the braided Hopf algebra structure of $B(V)$ is the only concern, we can omit some  Yetter-Drinfeld module information of $B(V)$ and realize it in a new Yetter-Drinfeld category. Let $B(V)\in \ _{\k G}^{\k G}\mathcal{YD}^\Phi$ and $B(U)\in \ _{\k H}^{\k H}\mathcal{YD}^\Psi$, we will say that $B(V)$ is isomorphic to $B(U)$ if there is a linear isomorphism $F\colon B(V)\To B(U)$ which preserves the multiplication and comultiplication.

For a Yetter-Drinfeld module $V\in \ _{\k G}^{\k G}\mathcal{YD}^\Phi$, we use $\delta_V\colon V\To \k G\otimes V$ to denote the comodule structure map of $V$.
\begin{lemma}[\cite{HLYY2}, Lemma 4.4]\label{l2.7}
Suppose $V\in \ _{\k G}^{\k G}\mathcal{YD}^\Phi$ and $U\in \ _{\k H}^{\k H}\mathcal{YD}^\Psi$, where $H$ is a finite abelian group. Let $G_V$ and $H_U$ be the support groups of $V$ and $U$ respectively. If there is a linear isomorphism $F\colon V\To U$ and 
a group epimorphism $\pi \colon G_V\To H_U$ such that:
\begin{eqnarray*}
&&\delta_U\circ F=(\pi\times F)\circ \delta_V,\\ 
&&F(g \triangleright v)=\pi(g) \triangleright F(v),\\ 
&&\Phi|_{G_V}=(\pi^*\Psi)|_{H_U} 
\end{eqnarray*}
for any $g\in G_V, \ v\in V$. Then $B(V)$ is isomorphic to $B(U)$.
\end{lemma} 

Let $G$ and $\mathbb{G}$ be two finite groups and $\pi\colon\mathbb{G}\to G$ a group epimorphism, $\iota\colon G\to \mathbb{G}$ be a section of $\pi$, that is $\pi\circ \iota=\id_G$. With these notations, we have following lemma.
\begin{lemma}\label{l2.8}
Let $V\in {_{\k G}^{\k G}\mathcal{YD}^\Phi}$. Then there is an object $\widetilde{V}\in {_{\k \mathbb{G}}^{\k \mathbb{G}}\mathcal{YD}^{\pi^*{\Phi}}}$ such that $\widetilde{V}=V$ as linear spaces and the Yetter-Drinfeld module structure is determined by
\begin{eqnarray}
&&\delta_{\widetilde{V}}=(\iota\otimes \id)\circ \delta_{V},\label{eq3.1} \\ 
&&g \triangleright v=\pi(g) \triangleright v \label{eq3.2}
\end{eqnarray}
for any $g\in \mathbb{G},\ v\in V$. Moreover, we have $B(V)\cong B(\widetilde{V})$.
\end{lemma} 

\begin{proof}
We need to show that the space $\widetilde{V}$ with action and coaction of $\mathbb{G}$ defined by \eqref{eq3.1} and  \eqref{eq3.2} is a Yetter-Drinfeld module in ${_{\k \mathbb{G}}^{\k \mathbb{G}}\mathcal{YD}^{\pi^*{\Phi}}}$. Let $V=\bigoplus_{g\in G}V_g$. Then it is obvious that we have $\widetilde{V}=\bigoplus_{g\in G}\widetilde{V}_{\iota(g)}$ such that $\widetilde{V}_{\iota(g)}=V_g$ as vector spaces for all $g\in G$. We only need to prove that $\widetilde{V}_{\iota(g)}$ is a projective $\mathbb{G}$-representation associated to 
$2$-cocycle ${\pi^*{\Phi}}_{\iota(g)}$. Let $e,f\in \mathbb{G}$ and $v\in V$, we have 
\begin{eqnarray*}
&&e\triangleright (f\triangleright v)= \pi(e)\triangleright (\pi(f)\triangleright v)\\
&=&\widetilde{\Phi}_g(\pi(e),\pi(f)) \pi(ef)\triangleright v\\
&=&\frac{\Phi(g,\pi(e),\pi(f))\Phi(\pi(e),\pi(f),g)}{\Phi(\pi(e),g,\pi(f))} ef \triangleright v\\
&=&\frac{\Phi(\pi\circ\iota(g),\pi(e),\pi(f))\Phi(\pi(e),\pi(f),\pi\circ\iota(g))}{\Phi(\pi(e),\pi\circ\iota(g),\pi(f))} ef \triangleright v\\
&=&\frac{\pi^*\Phi(\iota(g),e,f)\pi^*\Phi(e,f,\iota(g))}{\pi^*\Phi(e,\iota(g),f)} ef \triangleright v\\
&=&\widetilde{\pi^*\Phi}_{\iota(g)}(e,f) ef \triangleright v.
\end{eqnarray*}
So we have proved that $\widetilde{V}_{\iota(g)}$ is a projective $\mathbb{G}$-representation associated to the
$2$-cocycle ${\pi^*{\Phi}}_{\iota(g)}$, and hence $\widetilde{V}\in {_{\k \mathbb{G}}^{\k \mathbb{G}}\mathcal{YD}^{\pi^*{\Phi}}}$.
The isomorphism $B(V)\cong B(\widetilde{V})$ follows from Lemma \ref{l2.7} immediately.
\end{proof}

Let $B(V)$ be a Nichols algebra in $^{\k G}_{\k G}\mathcal{Y}\mathcal{D}^\Phi$. By lemma \ref{l2.8}, $B(V)$ is isomorphic to a Nichols algebra $B(\widetilde{V})$ in $^{\k \widehat{G}}_{\k \widehat{G}}\mathcal{Y}\mathcal{D}^{\pi^*\Phi}$.  If $\Phi$ is an abelian $3$-cocycle, then there is a $2$-cochain $J$ of $\widehat{G}$ such that $\partial J=\pi^*\Phi$. According to Lemma \ref{l2.6}, $B(\widetilde{V}^{J^{-1}})$ is a Nichols algebra in $^{\k \widehat{G}}_{\k \widehat{G}}\mathcal{Y}\mathcal{D}$, which is twist equivalent to $B(V)$. So we obtain the following proposition.

\begin{proposition}\label{p2.9}
Let $\Phi$ be an abelian $3$-cocycle and $B(V)$ be a Nichols algebra in $^{\k G}_{\k G}\mathcal{Y}\mathcal{D}^\Phi$. Then $B(V)$ is twist equivalent to a Nichols algebra in $^{\k \widehat{G}}_{\k \widehat{G}}\mathcal{Y}\mathcal{D}$.
\end{proposition}

Let $V\in {^{\k G}_{\k G}\mathcal{Y}\mathcal{D}^\Phi}$, if the action of the support group $G_V$ on $V$ is diagonal, then $V$ has a basis $\{X_1,\cdots, X_n\}$ such that 
\begin{equation}\label{eq2.18}
\begin{split}
\delta(X_i)=g_i\otimes X_i,\ \ g_i\triangleright X_j=q_{ij}X_j,
\end{split}
\end{equation}
where $q_{ij}\in \k$ for $1\leq i, j\leq n$. Such a basis $\{X_1,\cdots, X_n\}$ is called a {\bf standard basis} of $V$.
The following lemma follows from \cite[Lemma 4.1]{HLYY2} immediately. 
\begin{lemma}\label{l2.9}
Let $V\in {^{\k G}_{\k G}\mathcal{Y}\mathcal{D}^\Phi}$, the following three conditions are equivalent:
\begin{itemize}
\item[(1).] $V$ has a standard basis.
\item[(2).]   The action of support group $G_V$ on $V$ is diagonal.
\item[(3).] $\Phi_{G_V}$ is an abelian $3$-cocycle on $G_V$.
\end{itemize}
\end{lemma}

Now suppose $V$ has a standard basis $\{X_1,\cdots, X_n\}$, then we can define a nondirected graph $\mathcal{D}(V)$ associated to $B(V)$ as follows:
\begin{itemize}
\item[1)] There is a bijecton $\phi$ from $I=\{ 1, 2, \dots, n \}$ to the set of vertices of $\mathcal{D}(V)$.
\item[2)] For all $1\leq i\leq n,$ the vertex $\phi(i)$ is labelled  by $q_{ii}.$
\item[3)] For all $1\leq i,j\leq n,$ the number  $n_{ij}$ of edges between $\phi(i)$ and $\phi(j)$ is either $0$ or $1.$ If $i=j$ or $q_{ij}q_{ji}=1$ then $n_{ij}=0,$ otherwise $n_{ij}=1$ and the edge is labelled by $\widetilde{q_{ij}}=q_{ij}q_{ji}$ for all $1\leq i<j\leq n.$
\end{itemize}
The diagram $\mathcal{D}(V)$ is called the {\bf generalized Dynkin diagram} of $B(V)$.  Note that a Nichols algebra of diagonal type always has a generalized Dynkin diagram. It is also helpful to point out that if the generalized Dynkin diagram $\mathcal{D}(V)$ exists, it does not depend on the choice of the standard basis of $V$.  It is not hard to see that if $B(V)$ has a generalized Dynkin diagram, then $B(V^J)$ also have the same generalized Dynkin diagram with $B(V)$ for any $2$-cochain $J$ of $G$. So combining this with Proposition \ref{p2.9}, we have the following important proposition.

\begin{proposition}\label{p2.11}
Let $B(V)\in {_{\k G}^{\k G}\mathcal{YD}^\Phi}$ be a Nichols algebra with a standard basis. Then $B(V)$ is twist equivalent to a Nichols algebra $B(U)$ in ${_{\k \widehat{G_V}}^{\k \widehat{G_V}}\mathcal{YD}}$, and the two Nichols algebras have the same generalized Dynkin diagrams.
\end{proposition}

According to this proposition, all finite-dimensional Nichols algebras with a standard basis can be determined by Heckenberger's classification result of arithmetic root systems \cite{H4}. Note that if $B(V)$ is rank $1$ or rank $2$, then $G_V$ must be a finite cyclic group or direct product of two finite cyclic groups. According to \eqref{eq2.4}, all the $3$-cocycles on finite cyclic group or direct product of two finite cyclic groups must be abelian. So a Nichols algebra of rank $1$ or rank $2$ always has a standard basis. One of the main result in \cite{HYZ} is as follows.
\begin{proposition}\cite[Proposition 3.18]{HYZ}\label{p5.3}
Suppose $V\in {_{\k G}^{\k G} \mathcal{YD}^\Phi}$ is a simple Yetter-Drinfeld module of nondiagonal type with $\deg V=g$. Then $B(V)$ is finite dimensional if and only if $V$ is one of the following two types:
\begin{itemize}
\item[(I).] $g\triangleright v=-v$ for all $v\in V$;
\item[(II).] $\dim(V)=2$ and $g\triangleright v=\zeta_3 v$ for all $v\in V$, here $\zeta_3$ is a $3$-rd primitive root of unity.
\end{itemize} 
\end{proposition}

\section{Nichols algebras}
In this section, we will study Nichols algebras  without a standard basis in ${_{\k G}^{\k G}\mathcal{YD}^\Phi}$, where $G$ is a finite abelian group and $\Phi$ is a $3$-cocycle on $G$.  The main result is as follows.

\begin{theorem}\label{t3.1}
Suppose that $B(V)\in {_{\k G}^{\k G}\mathcal{YD}^\Phi}$ has no standard basis, then $B(V)$ is infinite dimensional.
\end{theorem}

See Remark \ref{r3.14} for the proof.  Since a Nichols algebra of rank $1$ or $2$ always has a standard basis, so our start point will be Nichols algebras of rank $3$.

\subsection{Nichols algebras of rank $3$}
Suppose $B(V)\in {_{\k G}^{\k G} \mathcal{YD}^\Phi}$ is a Nichols algebra of rank $3$. If $\Phi_{G_V}$ is an abelian $3$-cocycle on $G_V$, then the dimension of $B(V)$ can be determined by Proposition \ref{p2.11}.  So in this subsection, we will mainly consider the case that $\Phi_{G_V}$ is nonabelian.

\begin{definition}
Let $\alpha$ be a $2$-cocycle on $G$.  An element $g\in G$ is called an $\alpha$-element if $\alpha(g,h)=\alpha(h,g)$ for all $h\in G$.
\end{definition}

\begin{lemma}\label{l3.3}
Suppose $\Phi$ is a $3$-cocycle on $G$, $g\in G$ , then $g$ is a $\widetilde{\Phi}_g$-element.
\end{lemma}
\begin{proof}
For each element $h\in G$, we have 
\begin{equation*}
\widetilde{\Phi}_g(g,h)=\frac{\Phi(g,g,h)\Phi(g,h,g)}{\Phi(g,g,h)}=\frac{\Phi(h, g,g)\Phi(g,h,g)}{\Phi(h,g,g)}=\widetilde{\Phi}_g(h,g).
\end{equation*}
\end{proof}

\begin{lemma}\label{l6.3}
Let $G$ be a finite abelian group and $\Phi$ a $3$-cocycle on $G$. Then for any $g_1,g_2,g_3\in G$, we have
\begin{equation}\label{e3.1}
\frac{\widetilde{\Phi}_{g_1}(g_2,g_3)}{\widetilde{\Phi}_{g_1}(g_3,g_2)}=\frac{\widetilde{\Phi}_{g_2}(g_3,g_1)}{\widetilde{\Phi}_{g_2}(g_1,g_3)}=\frac{\widetilde{\Phi}_{g_3}(g_1,g_2)}{\widetilde{\Phi}_{g_3}(g_2,g_1)}.
\end{equation}
\end{lemma}

\begin{proof}
By definition of $\widetilde{\Phi}_{g}$ (see \eqref{e2.5}), we have
\begin{equation*}
\frac{\widetilde{\Phi}_{g_1}(g_2,g_3)}{\widetilde{\Phi}_{g_1}(g_3,g_2)}=\frac{\Phi(g_1,g_2,g_3)\Phi(g_2,g_3,g_1)\Phi(g_3,g_1,g_2)}{\Phi(g_2,g_1,g_3)\Phi(g_1,g_3,g_2)\Phi(g_3,g_2,g_1)}.
\end{equation*}
Similarly, 
\begin{eqnarray*}
\frac{\widetilde{\Phi}_{g_2}(g_3,g_1)}{\widetilde{\Phi}_{g_2}(g_1,g_3)}=\frac{\widetilde{\Phi}_{g_3}(g_1,g_2)}{\widetilde{\Phi}_{g_3}(g_2,g_1)}=
\frac{\Phi(g_1,g_2,g_3)\Phi(g_2,g_3,g_1)\Phi(g_3,g_1,g_2)}{\Phi(g_2,g_1,g_3)\Phi(g_1,g_3,g_2)\Phi(g_3,g_2,g_1)}.
\end{eqnarray*}
Thus we obtain \eqref{e3.1}.
\end{proof}

Next, we consider the structure of simple twisted Yetter-Drinfeld modules of nondiagonal type.
\begin{lemma}\label{l3.5}
Assume $G=\langle g_1,g_2,g_3\rangle$ and $V\in  {_{\k G}^{\k G} \mathcal{YD}^\Phi}$ is a simple Yetter-Drinfeld module with $\deg V= g_1$. Then  
$\dim(V)=n$, where $n$ is the order of $\frac{\widetilde{\Phi}_{g_1}(g_2,g_3)}{\widetilde{\Phi}_{g_1}(g_3,g_2)}$. Moreover, there exists a basis $\{X_1,X_2,\cdots, X_n\}$ of $V$ such that 
\begin{eqnarray}
&&g_1\triangleright X_i=\alpha X_i,\  \ 1\leq i\leq n;\label{e6.5}\\
&&g_2 \triangleright X_i=\beta (\frac{\widetilde{\Phi}_{g_1}(g_2,g_3)}{\widetilde{\Phi}_{g_1}(g_3,g_2)})^{i-1} X_i,\ \ 1\leq i\leq n; \label{e6.6}\\
&& g_3 \triangleright X_i=X_{i+1}, \ g_3\triangleright X_n=\gamma X_1, \  \ 1\leq i\leq n-1.\label{e6.7}
\end{eqnarray}
Here $\alpha,\beta,\gamma\in \k^*$ satisfy
\begin{eqnarray}
\alpha^{m_1}&=&\prod_{i=1}^{m_1-1}\Phi_{g_1}(g_1, g_1^i),\\
\beta^{m_2}&=&\prod_{i=1}^{m_2-1}\Phi_{g_1}(g_2, g_2^i), \\
\gamma^{\frac{m_3}{n}}&=&\prod_{i=1}^{m_3-1}\Phi_{g_1}(g_3, g_3^i),
\end{eqnarray}
 where $m_i=|g_i|$  is the order of $g_i$ for $1\le i\le 3$.  In particular, $V$ is of diagonal type if and only if $\frac{\widetilde{\Phi}_{g_1}(g_2,g_3)} {\widetilde{\Phi}_{g_1}(g_3,g_2)} = 1$.
\end{lemma}
 \begin{proof}
Let $g\in G$ and $v\in V$, $m=|g|$, by \eqref{e2.5} we have 
\begin{equation*}
\underbrace{g\triangleright (g\triangleright(\cdots (g}_{m}\triangleright v)\cdots ))=\prod_{i=1}^{m-1}\widetilde{\Phi}_{g_1}(g,g^i) v.
\end{equation*}
So it is obvious that the action of each element of $G$ on $V$ is diagonal. Moreover, by Lemma \ref{l3.3}, for any $g\in G$ and $v\in V$ we have  
\begin{equation}
g_1\triangleright(g\triangleright v)=\widetilde{\Phi}_{g_1}(g_1,g)(g_1g)\triangleright v=\widetilde{\Phi}_{g_1}(g,g_1)(gg_1)\triangleright v=g\triangleright(g_1\triangleright v).
\end{equation}
The identity implies that the map 
\[
g_1\colon V\To V,  \ 
v\mapsto g_1\triangleright v
\]
 is an isomorphism of projective $G$-representations associated to $\Phi_{g_1}$. Since $V$ is irreducible, by Schur's Lemma we have 
$$g_1\triangleright v=\alpha v,\ \ \forall v\in V$$
for some scalar $\alpha \in \k^*$. Since $$\underbrace{g_1\triangleright (g_1\triangleright(\cdots (g_1}_{m_1}\triangleright v)\cdots ))=\prod_{i=1}^{m_1-1}\Phi_{g_1}(g_1, g_1^i) v,$$ we get \eqref{e6.5}.

Take $0\ne v\in V$ such that $g_2\triangleright v= \beta v$ for some $\beta\in \k^*$. Let $s$ be the minimal positive integer such that 
\begin{equation}
 \underbrace{g_3\triangleright(g_3\triangleright(\cdots( g_3}_{s}\triangleright  v)\cdots))=\gamma v
\end{equation}
for some $\gamma\in \k^*$. Note that such integer $s$ exits and $s| m_3$ since
\begin{equation}
 \underbrace{g_3\triangleright(g_3\triangleright(\cdots( g_3}_{m_3}\triangleright  v)\cdots))=\prod_{i=1}^{m_3-1}\widetilde{\Phi}_{g_1}(g_3,g_3^i) \ v.
 \end{equation}
Since $g_1\triangleright v=\alpha v$ and $V$ is an irreducible projective $G$-representation with respect to $\Phi_{g_1}$, $V$ must be spanned by 
$$\{v, g_3\triangleright v,\  g_3\triangleright( g_3\triangleright  v), \cdots, \underbrace{g_3\triangleright(g_3\triangleright(\cdots( g_3}_{s-1}\triangleright  v)\cdots))\}.$$ 

In fact, let $X_i= \underbrace{g_3\triangleright(g_3\triangleright(\cdots( g_3}_{i-1}\triangleright  v)\cdots)), 1\leq i\leq s$. Then we have $g_2 \triangleright X_1=\beta X_1$,  and for all $1\leq i\leq s$ we have 
\begin{equation}
\begin{split}
g_2 \triangleright X_i=&g_2 \triangleright (g_3 \triangleright X_{i-1})=\Phi_{g_1}(g_2,g_3) (g_2g_3) \triangleright X_{i-1}\\
=&\frac{\Phi_{g_1}(g_2,g_3) }{\Phi_{g_1}(g_3,g_2)}g_3 \triangleright (g_2 \triangleright X_{i-1}).
\end{split}
\end{equation}
So we get 
\begin{equation}
g_2  \triangleright X_i=(\frac{\Phi_{g_1}(g_2,g_3) }{\Phi_{g_1}(g_3,g_2)})^{i-1}\beta X_i, \ 1\leq i\leq s.
\end{equation}
 
We claim that $\dim V= s$. Let $\zeta_s$ be a primitive $s$-th roots of unit and $\epsilon_1$ be an $s$-th root of $\gamma$. Then  $\epsilon_1,\epsilon_2=\epsilon_1\zeta_s,\cdots,\epsilon_s=\epsilon_1\zeta_s^{s-1}$ are all $s$-th roots of $\gamma$. For all $1\leq i\leq s$, we set 
$$Y_i=X_1+\epsilon_i^{-1}X_2+\cdots + \epsilon_i^{1-l}X_l+\cdots +\epsilon_i^{1-s}X_s.$$
Then for all $1\leq i\leq s$ we have 
\begin{equation}
\begin{split}
g_3  \triangleright Y_i=&g_3  \triangleright (X_1+\epsilon_i^{-1}X_2+\cdots + \epsilon_i^{1-l}X_l+\cdots +\epsilon_i^{1-s}X_s)\\
=& X_2+\epsilon_i^{-1}X_3+\cdots + \epsilon_i^{1-l}X_{l+1} \cdots +\epsilon_i^{2-s}X_s+ \epsilon_i^{1-s}\gamma X_1)\\
=& \epsilon_i X_1+X_2+\epsilon_i^{-1}X_3+\cdots + \epsilon_i^{1-l}X_{l+1} \cdots +\epsilon_i^{s-2}X_s\\
=& \epsilon_i Y_i.
\end{split}
\end{equation}
Then $Y_i$'s are clearly linearly independent since they correspond to different eigenvalues. Then $\{Y_1, \cdots, Y_s\}$ forms a basis of $V$, and $\{\epsilon_i| 1\leq i\leq s\}$ are all eigenvalues of $g_3$ when viewed as a linear transformation on $V$.

Next we prove that $s=n$.  
Notice that 
\begin{equation}
\begin{split}
g_3\triangleright (g_2\triangleright Y_1)=&\Phi_{g_1}(g_3,g_2) (g_3g_2)  \triangleright Y_1\\
=&\frac{\Phi_{g_1}(g_3,g_2) }{\Phi_{g_1}(g_2,g_3) } g_2\triangleright (g_3\triangleright Y_1)\\
=&\frac{\Phi_{g_1}(g_3,g_2) }{\Phi_{g_1}(g_2,g_3) }\epsilon_1 g_2\triangleright  Y_1.
\end{split}
\end{equation}
Inductively we have 
\begin{equation}
\begin{split}
&g_3\triangleright ( \underbrace{g_2\triangleright(g_2\triangleright(\cdots( g_2}_{i}\triangleright  Y_{1})\cdots)))\\
&=\Big(\frac{\Phi_{g_1}(g_3,g_2) }{\Phi_{g_1}(g_2,g_3) }\Big)^i\epsilon_1  \underbrace{g_2\triangleright(g_2\triangleright(\cdots( g_2}_{i}\triangleright  Y_{1})\cdots)).
\end{split}
\end{equation}
So we have $g_3\triangleright ( \underbrace{g_2\triangleright(g_2\triangleright(\cdots( g_2}_{n}\triangleright  Y_{1})\cdots)))=\epsilon_1  \underbrace{g_2\triangleright(g_2\triangleright(\cdots( g_2}_{n}\triangleright  Y_{1})\cdots))$, and hence 
$$ \underbrace{g_2\triangleright(g_2\triangleright(\cdots( g_2}_{n}\triangleright  Y_{1})\cdots)))=k Y_1$$ for some scaler $k\in \k^*$. This implies that $Y_1, g_2\triangleright Y_1,\cdots , \underbrace{g_2\triangleright(g_2\triangleright(\cdots( g_2}_{n-1}\triangleright  Y_{1})\cdots))$ span a sub-Yetter-Drinfeld module of $V$, and hence $V$ since $V$ is simple. So we obtain that $\dim(V)=n$ and hence $s=n$.

 The equations 
\eqref{e6.6} and \eqref{e6.7} follow from 
\begin{eqnarray*}
\underbrace{g_2\triangleright (g_2\triangleright(\cdots (g_2}_{m_2}\triangleright X_1)\cdots ))&=&\prod_{i=1}^{m_2-1}\Phi_{g_1}(g_2, g_2^i) X_1,\\
\underbrace{g_3\triangleright (g_3\triangleright(\cdots (g_3}_{m_3}\triangleright X_1)\cdots ))&=&\prod_{i=1}^{m_3-1}\Phi_{g_1}(g_3, g_3^i) X_1.
\end{eqnarray*}

The last statement is obvious and the  proof is completed.
\end{proof}

By this lemma, we have the following important proposition.
\begin{proposition}\label{p3.6}
Let $V=V_1\oplus V_2\oplus V_3\in {_{\k G}^{\k G} \mathcal{YD}^\Phi}$ be a direct sum of simple twisted Yetter-Drinfeld modules with $G=G_V$. Then we have 
\begin{equation}\label{e6.15}
\dim(V_1)=\dim(V_2)=\dim(V_3).
\end{equation}
\end{proposition}

\begin{proof}
Let $\deg V_i=g_i$, $1\leq i\leq 3$. Then $G=\langle g_1, g_2, g_3\rangle$ since  $G_V=G$. By Lemma \ref{l3.5}, we have $\dim (V_1)=\Big| \frac{\widetilde{\Phi}_{g_1}(g_2,g_3)}{\widetilde{\Phi}_{g_1}(g_3,g_2)}\Big|$,
 $\dim (V_2)=\Big| \frac{\widetilde{\Phi}_{g_2}(g_3,g_1)}{\widetilde{\Phi}_{g_2}(g_1,g_3)}\Big|$,  $\dim (V_3)=\Big| \frac{\widetilde{\Phi}_{g_3}(g_1,g_2)}{\widetilde{\Phi}_{g_3}(g_2,g_1)}\Big|$, and \eqref{e6.15} follows from Lemma \ref{l6.3}.
\end{proof}

Now we can consider nondiagonal Nichols algebras of rank $3$. Firstly, we have the following propositions.
\begin{proposition}\label{p3.7}
Let $V=V_1\oplus V_2\oplus V_3\in {_{\k G}^{\k G} \mathcal{YD}^\Phi}$ be a direct sum of simple twisted Yetter-Drinfeld modules with $G=G_V$. If $\dim (V_1)=\dim (V_2)=\dim (V_3)\geq 3$, then $B(V)$ is infinite dimensional.
\end{proposition}
\begin{proof}
Let $\deg V_i= g_i$, $1\leq i\leq 3$. Since $G=G_V$, we have $G=\langle g_1,g_2,g_3\rangle$. Let $|g_i|=m_i, 1\leq i\leq 3$, and $n=\Big|\frac{\widetilde{\Phi}_{g_1}(g_2,g_3)}{\widetilde{\Phi}_{g_1}(g_3,g_2)}\Big|$. Then $\dim (V_1)=\dim (V_2)=\dim (V_3)=n$ by Proposition \ref{p3.6}. According to Proposition \ref{p5.3}, if $V_i$ is not a simple Yetter-Drinfeld module of type (I) for some $i\in \{1,2, 3\}$, then $B(V_i)$ must be infinite dimensional, hence $B(V)$ is infinite dimensional. 

In the following, we assume that $V_1,V_2,V_3$ are simple Yetter-Drinfeld modules of type (I). By Lemma \ref{l3.5}, $V_1$ has a basis $\{X_1, X_2,\cdots, X_n\}$ such that 
\begin{eqnarray}
&&g_1\triangleright X_i=- X_i,\  \ 1\leq i\leq n,\label{eq3.16}\\
&&g_2 \triangleright X_i=\beta_1 (\frac{\widetilde{\Phi}_{g_1}(g_2,g_3)}{\widetilde{\Phi}_{g_1}(g_3,g_2)})^{i-1} X_i,\ \ 1\leq i\leq n,
\end{eqnarray}
where $\beta_1^{m_2}=\prod_{i=1}^{m_2-1}\Phi_{g_1}(g_2, g_2^i)$. Here \eqref{eq3.16} follows from the fact that $V_1$ is a simple Yetter-Drinfeld module of type (I).

Similarly, $V_2$ also has a basis $\{Y_1,Y_2,\cdots, Y_n\}$ such that 
\begin{eqnarray}
&&g_2\triangleright Y_i=- Y_i,\  \ 1\leq i\leq n,\\
&&g_1 \triangleright Y_i=\beta_2 (\frac{\widetilde{\Phi}_{g_2}(g_1,g_3)}{\widetilde{\Phi}_{g_2}(g_3,g_1)})^{i-1} Y_i,\ \ 1\leq i\leq n,
\end{eqnarray}
where $\beta_2^{m_1}=\prod_{i=1}^{m_1-1}\Phi_{g_2}(g_1, g_1^i)$.
Let $H=G_{V_1\oplus V_2}$ and $\Psi=\Phi_H$. Since $H$ is direct sum of two cyclic groups, $\Psi$ must be an abelian $3$-cocycle on $H$ by Lemma \ref{l2.5}. This implies $V_1\oplus V_2$ is a Yetter-Drinfeld module of diagonal type in ${_{\k H}^{\k H} \mathcal{YD}^\Psi}$.
In the following, let  $W=\k \{X_1,X_2,X_3,Y_1,Y_2,Y_3\}$ be a submodule of $V_1\oplus V_2\in {_{\k H}^{\k H} \mathcal{YD}^\Psi}.$ We will consider the generalized Dynkin diagram $\mathcal{D}(W)$.

If $\beta_1\beta_2\neq (\frac{\widetilde{\Phi}_{g_2}(g_1,g_3)}{\widetilde{\Phi}_{g_2}(g_3,g_1)})^k$ for $k\in \{0,\pm1, \pm 2\}$, then the generalized Dynkin diagram $\mathcal{D}(W)$ (with unlabeled edges) associated to $B(W)$ is
\[ {\setlength{\unitlength}{1mm}
\THexagon{}{$-1$}{$-1$}{$-1$}{$-1$}{$-1$}{$-1$}} \quad .\]
By Propsition \ref{p2.11}, $B(W)$ is twist equivalent to a Nichols algebra $B(U)$  in ${_{\k \widehat{H}}^{\k \widehat{H}}\mathcal{YD}}$, and $B(U)$ has the same generalized Dynkin diagram with $B(W)$. Comparing the classification of generalized Dynkin diagrams of finite-dimensional Nichols algebras in \cite{H4}, $B(U)$ and hence $B(W)$ must be infinite dimensional. This implies that $B(V_1\oplus V_2)$ is infinite dimensional.

If $\beta_1\beta_2=(\frac{\widetilde{\Phi}_{g_2}(g_1,g_3)}{\widetilde{\Phi}_{g_2}(g_3,g_1)})^k$ for some $k\in \{0,\pm 1,\pm 2\}$, then the generalized Dynkin diagram  $\mathcal{D}(W)$ (with unlabeled edges) has a subdiagram of the form
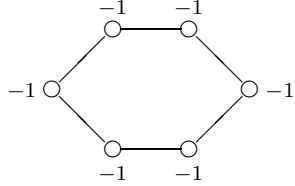
\[ {\setlength{\unitlength}{1mm}
\Hexagon{}{$-1$}{$-1$}{$-1$}{$-1$}{$-1$}{$-1$}} \quad .\]
Comparing the classification result in \cite{H4}, again we have $B(W)$ is infinite dimensional, and hence $B(V_1\otimes V_2)$ is infinite dimensional. 

In either case $B(V_1\oplus V_2)$ is infinite dimensional, hence so is $B(V)$ since $B(V_1\oplus V_2)$ is a subalgebra of $B(V)$.
\end{proof}

\begin{proposition}\label{p3.8}
Let $V=V_1\oplus V_2\oplus V_3\in {_{\k G}^{\k G} \mathcal{YD}^\Phi}$ be a direct sum of simple twisted Yetter-Drinfeld modules and $G=G_V$. Assume that $\dim (V_1)=\dim (V_2)=\dim (V_3)=2$, and at least one of $V_i$'s is of 
type (II). Then $B(V)$ is infinite dimensional.
\end{proposition}
\begin{proof}
Without loss of generality, we can assume that $V_1$ is a simple twisted Yetter-Drinfeld module of type (II). In what follows we will prove that $B(V_1\oplus V_2)$ (similarly for $B(V_1\oplus V_3)$) is infinite dimensional, which forces $B(V)$ is infinite dimensional.

Let $g_i=\deg(V_i)$ and $m_i=|g_i|$ for $1\leq i\leq 3$, $H:=G_{V_1\oplus V_2}=\langle g_1, g_2\rangle$ and $\Psi=\Phi|_{H}$. Then $\Psi$ is an abelian $3$-cocycle on $H$ by Lemma \ref{l2.5}, and hence $B(V_1\oplus V_2)$ is a Nichols algebra of diagonal type in  ${_{\k H}^{\k H} \mathcal{YD}^\Psi}$. By Lemma \ref{l3.5}, $V_1$ has a basis $\{X_1, X_2\}$ such that 
\begin{eqnarray}
&&g_1\triangleright X_i=\zeta_3 X_i, \ i=1,2;\\
&&g_2\triangleright X_1=\beta_1 X_1,\ g_2\triangleright X_2=-\beta_1 X_2.
\end{eqnarray}
Here $\beta_1$ is a root of unit satisfying $\beta_1^{m_2}=\prod_{i=1}^{m_2-1}\Phi_{g_1}(g_2, g_2^i)$. In the following, we need to consider two cases: (a) $V_2$ has type (I), (b) $V_2$ has type (II).

(a). Since $\dim(V_2)=2$, by Lemma \ref{l3.5},  $V_2$ has a basis $\{Y_1, Y_2\}$ such that 
\begin{eqnarray}
&&g_2\triangleright Y_i=- Y_i, \ i=1,2;\\
&&g_1\triangleright Y_1=\beta_2 Y_1,\ g_2\triangleright Y_2=-\beta_2 Y_2.
\end{eqnarray}
Here $\beta_2^{m_1}=\prod_{i=1}^{m_1-1}\Phi_{g_2}(g_1, g_1^i)$. 
If $\beta_1\beta_2\neq \pm 1$, the generalized Dynkin diagram  $\mathcal{D}(V_1\oplus V_2)$ (with unlabeled edges) of $B(V_1\oplus V_2)$ is
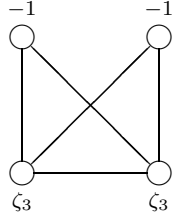
\[ {\setlength{\unitlength}{1.5mm}
\Tchainfour{}{$\zeta_3$}{$\zeta_3$}{$-1$}{$-1$}} \quad .\]
 If $\beta_1\beta_2= \pm 1$, the generalized Dynkin diagram of $\mathcal{D}(V_1\oplus V_2)$ is
\[ {\setlength{\unitlength}{1.5mm}
\Dchainfour{}{$-1$}{$-1$}{$\zeta_3$}{$\zeta_3^2$}{$\zeta_3$}{$-1$}{$-1$}} \quad .\]
By Propsition \ref{p2.11}, $B(V_1\oplus V_2)$ is twist equivalent to some Nichols algebra $B(U)$  in ${_{\k \widehat{H}}^{\k \widehat{H}}\mathcal{YD}}$, and $B(U)$ have the same generalized Dynkin diagram with $B(V_1\oplus V_2)$. By checking up the classification of generalized Dynkin diagrams of finite-dimensional Nichols algebras of diagonal type in \cite{H4},  we can see that $B(U)$ and hence $B(V_1\oplus V_2)$ is infinite dimensional.

(b). By Lemma \ref{l3.5},  $V_2$ has a basis $\{Y_1, Y_2\}$ such that 
\begin{eqnarray}
&&g_2\triangleright Y_i=\zeta_3^k Y_i, \ i=1,2; \\
&&g_1\triangleright Y_1=\beta_2 Y_1,\ g_2\triangleright Y_2=-\beta_2 Y_2.
\end{eqnarray}
Here $\beta_2^{m_1}=\prod_{i=1}^{m_1-1}\Phi_{g_2}(g_1, g_1^i)$ and $k\in \{1,2\}$ is a fixed number. 
If $\beta_1\beta_2\neq \pm 1$, then the generalized Dynkin diagram $\mathcal{D}(V_1\oplus V_2)$ (with unlabeled edges) of $B(V_1\oplus V_2)$ is 
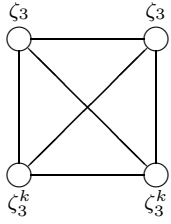
\[ {\setlength{\unitlength}{1.5mm}
\Fchainfour{}{$\zeta_3^k$}{$\zeta_3^k$}{$\zeta_3$}{$\zeta_3$}} \quad .\]
If $\beta_1\beta_2= \pm 1$, then the generalized Dynkin diagram $\mathcal{D}(V_1\oplus V_2)$ (with unlabeled edges) of $B(V_1\oplus V_2)$ is
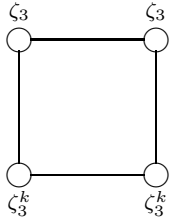
\[ {\setlength{\unitlength}{1.5mm}
\Echainfour{}{$\zeta_3^k$}{$\zeta_3^k$}{$\zeta_3$}{$\zeta_3$}} \quad .\]
In both cases $B(V_1\oplus V_2)$ is infinite dimensional by the classification of  generalized Dynkin diagrams of finite-dimensional Nichols algebras of diagonal type in \cite{H4}. 
\end{proof}

It remains to consider $B(V_1\oplus V_2\oplus V_3)$ where $V_1,V_2,V_3$ are simple twisted Yetter-Drinfeld modules of type (I) and $\dim(V_i)=2$, $1\leq i\leq 3$. We have the following theorem.

\begin{theorem}\label{t3.9}
Let $V_1, V_2, V_3\in  {_{\k G}^{\k G} \mathcal{YD}^\Phi}$ be simple Yetter-Drinfeld modules of type (I) such that $\dim(V_i)=2$,  $\deg(V_i)=g_i, 1\leq i\leq 3$ and $G=\langle g_1\rangle \times \langle g_2\rangle \times g_3\rangle$. Then the Nichols algebra $B(V_1\oplus V_2\oplus V_3)$ is infinite dimensional.
\end{theorem}

The proof of the theorem is quite technical and lengthy. To avoid digressing from the present main theme, we postpone the proof to the next section. With the help of this theorem, we can prove the following proposition.

\begin{proposition}\label{p3.10}
Let $V=V_1\oplus V_2\oplus V_3\in {_{\k G}^{\k G} \mathcal{YD}^\Phi}$ be a direct sum of simple Yetter-Drinfeld modules and $G=G_V$. If $\dim (V_1)=\dim (V_2)=\dim (V_3)=2 $ and $V_1, V_2, V_3$ are all simple Yetter-Drinfeld modules of 
type (I), then $B(V)$ is infinite dimensional.
\end{proposition}
\begin{proof}
Let $g_i=\deg(V_i)$ and $m_i=|g_i|$ for $i=1,2,3$. Let $\mathbbm{G}=\langle \mathbbm{g}_1\rangle \times \langle \mathbbm{g}_2\rangle \times \langle \mathbbm{g}_3\rangle$ be the abelian group with free generators  $\mathbbm{g}_1, \mathbbm{g}_2, \mathbbm{g}_3$ such that $|\mathbbm{g}_i|=m_i$ for $1\leq i\leq 3$. Then it is obvious that there is a group epimorphism $\pi\colon\mathbbm{G}\To G$ such that $\pi(\mathbbm{g}_i)=g_i$, $1\leq i\leq 3$. Let $\iota\colon G\To \mathbbm{G}$ be a section of $\pi$ ( that is $\pi\circ \iota=\id_G$) such that $\iota(\mathbbm{g}_i)=g_i$ for all $1\leq i\leq 3$. For each $i\in \{1,2,3\}$, let $\widetilde{V_i}\in {_{\k \mathbb{G}}^{\k \mathbb{G}}\mathcal{YD}^{\pi^*{\Phi}}}$ be the Yetter-Drinfeld module associated to $V_i$ define by \eqref{eq3.1}-\eqref{eq3.2}. Let $\widetilde{V}=\widetilde{V_1}\oplus \widetilde{V_2}\oplus \widetilde{V_3}$. By Lemma \ref{l2.8}, we have $B(V)\cong B(\widetilde{V})$. On the other hand, by Theorem \ref{t3.9}, $B(\widetilde{V})$ is infinite dimensional, thus $B(V)$ is also infinite dimensional.
\end{proof}

 Combining Proposition \ref{p3.7}, \ref{p3.8} and \ref{p3.10}, the following theorem is clear.
\begin{theorem}\label{t3.11}
Let $B(V)\in  {_{\k G}^{\k G} \mathcal{YD}^\Phi}$ be a nondiagonal Nichols algebra of rank $3$ with $G_V=G$. Then $B(V)$ is infinite dimensional.
\end{theorem}

\begin{corollary}\label{c3.12}
Let $B(V)\in  {_{\k G}^{\k G} \mathcal{YD}^\Phi}$ be a Nichols algebra of rank $3$, and $\Phi_{G_V}$ a nonabelian $3$-cocycle on $G_V$. Then $B(V)$ is infinite dimensional.
\end{corollary}
\begin{proof}
Let $H=G_V$ and $\Psi=\Phi_H$. Then the Nichols algebra $B(V)$ can be realized in  ${_{\k H}^{\k H} \mathcal{YD}^\Psi}$, and the rank of $B(V)$ must be greater than or equal to $3$. 
For  each $i\in \{1,2,3\}$, let $U_i$ be a nonzero simple Yetter-Drinfeld submodule of $V_i$ and $U=U_1\oplus U_2\oplus U_3$. It is clear $H=G_U$. By Lemma \ref{l2.9}, $U$ is nondiagonal  since $\Psi$ is nonabelian on $H=G_U$, thus $U_1,U_2,U_3$ are all nondiagonal in ${_{\k H}^{\k H} \mathcal{YD}^\Psi}$. Now $B(U)$ is infinite dimensional by Theorem \ref{t3.11},  and so is $B(V)$ since $B(U)\subset B(V)$.
\end{proof}

\subsection{The general case}
In this subsection, we will give a classification of finite-dimensional Nichols algebras in $^{\k G}_{\k G} \mathcal{YD}^{\Phi}$, where $G$ is a finite abelian group and $\Phi$ is a $3$-cocycle on $G$. Firstly, we have the following theorem.
\begin{theorem}\label{t3.13}
Let $B(V)\in  {_{\k G}^{\k G} \mathcal{YD}^\Phi}$ such that $\Phi_{G_V}$ is nonabelian. Then $B(V)$ is infinite dimensional.
\end{theorem}
\begin{proof}
By Lemma \ref{l2.7}, $B(V)$ can be realized in  $ {_{\k G_V}^{\k G_V} \mathcal{YD}^{\Phi|_{G_V}}}$. Since each $3$-cocycle of cyclic group or direct sum of two cyclic groups is abelian, so the rank of $B(V)$ must be greater than or equal to $3$ because ${\Phi|_{G_V}}$ is nonabelian. Thus $B(V)$ is infinite dimensional by Corollary \ref{c3.12}.
\end{proof}
\begin{remark}\label{r3.14} 
\emph{Since the conditions ``$\Phi|_{G_V}$ is abelian" and ``$B(V)$ has a standard basis" are equivalent by Lemma \ref{l2.9}, Theorem \ref{t3.13} is equivalent to Theorem \ref{t3.1}.}
\end{remark}
We draw the following immediate consequences.
\begin{corollary}\label{c3.14}
Suppose $B(V)\in  {_{\k G}^{\k G} \mathcal{YD}^\Phi}$ is a finite-dimensional Nichols algebra and $G=G_V$. Then $B(V)$ must be of diagonal type.
\end{corollary}
\begin{corollary}\label{c3.15}
Suppose $B(V)\in  {_{\k G}^{\k G} \mathcal{YD}^\Phi}$ is a finite-dimensional Nichols algebra. Then we have 
\begin{itemize}
\item[(1).] $\Phi|_{G_V}$ is abelian and $V$ has a standard basis;
\item[(2).] $B(V)$ is isomorphic to a Nichols algebra of diagonal type in ${_{\k G_V}^{\k G_V} \mathcal{YD}^{\Phi_{G_V}}}$.
\end{itemize}
\end{corollary}
\begin{proof}
(1) follows from Theorem \ref{t3.13}, and (2) follows from Proposition \ref{p3.10}.
\end{proof}

\subsection{Classification}
 Next we will present a classification of finite-dimensional Nichols algebras in  ${_{\k G}^{\k G} \mathcal{YD}^\Phi}$. 
%
 Let $\Delta_{\chi,E}$ be an arithmetic root system, where $E=\{e_1,\cdots, e_n\}$ is a set of free generator of $\Z^n$ and $\chi$ is a bicharacter of $\Z^n$. For each positive root $\alpha\in \Delta_{\chi,E}$, define
$q_\alpha=\chi(\alpha,\alpha)$.
Then the height of $\alpha$ is defined by
\begin{equation}
\operatorname{ht}(\alpha)=\left\{
              \begin{array}{ll}
                |q_\alpha|, & \hbox{if $q_\alpha\neq 1$ is a root of unity;} \\
                \infty, & \hbox{otherwise.}
              \end{array}
            \right.
\end{equation}

Let $V\in  {_{\k G}^{\k G} \mathcal{YD}^\Phi}$ be a Yetter-Drinfeld module with a standard basis  $\{Y_1,\cdots, Y_n\}$, $\deg(Y_i)=g_i, 1\leq i\leq n$. Then there is a pair $(\chi, E)$ associated to $V$, where $E=\{e_1,\cdots, e_n\}$ is a set of free generator of $\Z^n$ and $\chi$ is a bicharacter of $\Z^n$ given by 
\begin{equation}
g_i \triangleright Y_j = \chi(e_i,e_j)Y_j, \ 1\leq i,j\leq n.
\end{equation}
\begin{remark}
\emph{If $V$ has another standard basis $\{Y'_1,\cdots, Y'_n\}$, and a pair $(\chi', E')$ associated to $V$ with respect to the basis $\{Y'_1,\cdots, Y'_n\}$. Then there is an isomorphism $\tau\colon\Z^n\to \Z^n$ such that
\begin{equation*}
\tau(E)=E',\ \ \ \ \chi'(\tau(e_i),\tau(e_j))=\chi(e_i,e_j)
\end{equation*} for all $1\leq i, j\leq n$.}
\end{remark}

\begin{definition}
A Yetter-Drinfeld module $V\in  {_{\k G}^{\k G} \mathcal{YD}^\Phi}$ is said to be of {\bf finite type} if 
\begin{itemize}
\item[(1)] $V$ has a standard basis;
\item[(2)] $\Delta_{\chi,E}$ is an arithmetic root system and $\operatorname{ht}(\alpha)< \infty $ for all $\alpha\in \Delta_{\chi,E}$, where $(\chi, E)$ is a pair associated to $V$.
\end{itemize}
\end{definition}

\begin{theorem}\label{t8.8}
Let $V\in  {_{\k G}^{\k G} \mathcal{YD}^\Phi}$ be a Yetter-Drinfeld module.  Then $B(V)$ is finite dimensional if and only if $V$ is of finite type.
\end{theorem}
\begin{proof}
Firstly, suppose that $B(V)$ is finite dimensional. By Corollary \ref{c3.15}, $\Phi|_{G_V}$ is an abelian $3$-cocycle on $G_V$, thus $V$ has a standard basis. Let $H=G_V$, $\Psi=\Phi|_{G_V}$, and $(\chi, E)$ a pair associated to $V$. By Proposition \ref{p2.11},  $B(V)$ is isomorphic to a Nichols algebra $B(U)$ in $_{\k \widehat{H}}^{\k \widehat{H}} \mathcal{YD}$.  Note that $U$ can be obtained from $V$ by change of based groups and twisting, which do not change the pair $(\chi, E)$ associted to $V$, thus $(\chi, E)$ is also a pair associated to $U$. On the other hand, since $B(U)\in {_{\k H}^{\k H} \mathcal{YD}}$ is finite dimensional,  $\Delta_{\chi,E}$ is an arithmetic root system. Moreover, $B(U)$ is finite dimensional implies that the nilpotent index of a root vector of each root $\alpha\in \Delta_{\chi,E}$ is finite, i.e., $\operatorname{ht}(\alpha)< \infty $ for all $\alpha\in \Delta_{\chi,E}$.

Conversely, Suppose that $V$ is finite type. By Proposition \ref{p2.11}, $B(V)$ is twist equivalent to a Nichols algebra $B(U)$ in ${_{\k \widehat{H}}^{\k \widehat{H}} \mathcal{YD}}$, and $B(V)$ have the same generalized Dynkin diagram with $B(V)$. So $\Delta_{\chi,E}$ is the arithmetic root system of $B(U)$. Furthermore, $B(U)$ is finite dimensional since $\operatorname{ht}(\alpha)< \infty $ for all $\alpha\in \Delta_{\chi,E}$. This implies $B(V)$ is infinite dimensional since $B(V)$ is twist equivalent to $B(U)$.
\end{proof}

\section{The proof of Theorem \ref{t3.9}}
The main task of this section is to prove Theorem \ref{t3.9}. Firstly, we will consider a special case.

\subsection{A special case}
In this subsection, we will prove the following proposition.
\begin{proposition}\label{p7.1}
Suppose $G=\langle e\rangle \times \langle f\rangle \times \langle g\rangle$ is a finite abelian group, $\Phi$ is a $3$-cocycle on $G$ given by 
\begin{equation}
\Phi(e^{i_{1}}f^{i_{2}}g^{i_{3}},e^{j_{1}}f^{j_{2}}g^{j_{3}},e^{k_{1}}f^{k_{2}}g^{k_{3}})=(-1)^{i_1j_2k_3}
\end{equation}
for all $0\leq i_1,j_1, k_1<m_1,\ 0\leq i_2, j_2, k_2<m_1,\ 0\leq i_3, j_3, k_3<m_3$, where $m_1=|e|, m_2=|f|, m_3=|g|$. 
Let $U,V,W\in {_{\k G}^{\k G} \mathcal{YD}^\Phi}$ be simple Yetter-Drinfeld modules of type (I) such that $\deg(U)=e,\ \deg(V)=g,\  \deg(W)=f$ and $\dim(U)=\dim(V)=\dim(W)=2$. Then $B(U\oplus V\oplus W)$ is infinite dimensional.
\end{proposition}

In what follows, the Yetter-Drinfeld category ${_{\k G}^{\k G} \mathcal{YD}^\Phi}$ and modules $U,V,W$ in it are assumed satisfying the conditions of this proposition.

\begin{lemma}\label{l7.2}
The simple module $U$ has a basis $\{X_1,X_2\}$ satisfying 
\begin{eqnarray}
&&e\triangleright X_i=-X_i, \ i=1,2,\label{eq7.2}\\
&&f\triangleright X_1=\beta_1 X_1,  f\triangleright X_2=-\beta_1 X_2,\label{eq7.3}\\
&&g\triangleright X_1=\gamma_1 X_2,  g\triangleright X_2=\gamma_1 X_1.\label{eq7.4}
\end{eqnarray}
Here $\beta_1,\gamma_1\in \k$ such that $\beta_1^{m_2}=1, \ \gamma_1^{m_3}=1$.
\end{lemma}
\begin{proof}
Since $U$ is of type (I) and $\dim(U)=2$, by Lemma \ref{l3.5}, $U$ has a basis $\{X'_1, X'_2\}$ such that 
\begin{eqnarray*}
&&e\triangleright X'_i=-X'_i, \ i=1,2,\\
&&f\triangleright X'_1=\beta_1 X'_1,  f\triangleright X_2=-\beta_1 X_2,\\
&&g\triangleright X'_1=X'_2,  g\triangleright X'_2=\gamma'_1 X'_1.
\end{eqnarray*}
Here $\beta_1^{m_2}=1, {\gamma'_1}^{\frac{m_3}{2}}=1.$
Let $\gamma\in \k$ such that $\gamma_1^2=\gamma'_1$. It is clear that $\gamma_1^{m_3}=1$.  Let $X_1=X'_1,\ X_2=\frac{1}{\gamma_1}X'_2$, we get \eqref{eq7.2}-\eqref{eq7.4}.
\end{proof}

Similar to Lemma \ref{l7.2}, we have the following two lemmas.

\begin{lemma}
The simple module $V$ has a basis $\{Y_1,Y_2\}$ satisfying 
\begin{eqnarray}
&&f\triangleright Y_i=-Y_i, \ i=1,2,\label{eq7.8}\\
&&g\triangleright Y_1=\beta_2 Y_1,  g\triangleright Y_2=-\beta_2 Y_2,\label{eq7.9}\\
&&e\triangleright Y_1=\gamma_2 Y_2,  e\triangleright Y_2=\gamma_2 Y_1.\label{eq7.10}
\end{eqnarray}
Here $\beta_2,\gamma_2\in \k$ are numbers such that $\beta_2^{m_3}=1, \ \gamma_2^{m_1}=1$.
\end{lemma}

\begin{lemma}
The simple module $W$ has a basis  $\{Z_1,Z_2\}$ such that 
\begin{eqnarray}
&&g\triangleright Z_i=-Z_i, \ i=1,2,\label{eq7.11} \\
&&f\triangleright Z_1=\beta_3 Z_1,  f\triangleright Z_2=-\beta_3 Z_2,\label{eq7.12}\\
&&e\triangleright Z_1=\gamma_3 Z_2,  e\triangleright Z_2=\gamma_3 Z_1,\label{eq7.13}
\end{eqnarray}
where $\beta_3,\gamma_3\in \k$ are numbers $\beta_3^{m_2}=1, \gamma_3^{m_1}=1.$
\end{lemma}

In the following, $(\beta_i, \gamma_i),\ i=1,2,3$ will be called {\bf structure constants} of $U,V,W$ respectively. It is clear that the structure constants depend on the choice of the bases of $U,V,W$.
\begin{remark}\label{r5.7}
The structure constants $(\beta_i, \gamma_i), i=1,2,3$ can be changed to be $(-\beta_i, \gamma_i),$ $(\beta_i, -\gamma_i)$ and $(-\beta_i, -\gamma_i)$ if we transform the bases of $U,V$ and $W$ respectively.
For example, let $\overline{X}_1=X_2, \overline{X}_2=X_1,$ the constants will changed to be $(-\beta_1, \gamma_1)$. Let $\overline{X}_1=X_1, \overline{X}_2=-X_2,$ the constants will be  $(\beta_1, -\gamma_1)$. Let
$\overline{X}_1=X_2, \overline{X}_2=-X_1,$ the constants will changed to be  $(-\beta_1, -\gamma_1)$.
 \end{remark}

\begin{proposition} \label{p7.6}
Keep the previous notations, we have
\begin{itemize}
\item[(1).] If $\beta_2\beta_3\neq \pm 1$, then $B(V\oplus W)$ is infinite dimensional.
\item[(2).] If $\beta_1\gamma_2\neq \pm 1$, then $B(U\oplus V)$ is infinite dimensional.
\item[(3).] If $\gamma_1\gamma_3\neq \pm 1$, then $B(U\oplus W)$ is infinite dimensional.
\end{itemize}
\end{proposition}
\begin{proof}
(1). Let $H=\langle f\rangle \times \langle g\rangle$ and $\Psi=\Phi|_{H}$. Then it is obvious that $\Psi=1$ and $B(V\oplus W)\in {_{\k H}^{\k H} \mathcal{YD}}$ is a Nichols algebra of diagonal type. 
The generalized Dynkin diagram $\mathcal{D}(V\oplus W)$ is 
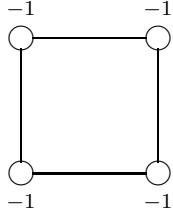
\[ {\setlength{\unitlength}{1.5mm}
\Echainfour{}{$-1$}{$-1$}{$-1$}{$-1$}} \quad .\]
Comparing the classification result of finite-dimensional Nichols algebras in ${_{\k H}^{\k H} \mathcal{YD}}$, we obtain that $B(V\oplus W)$ is infinite dimensional.

(2). Let $\overline{Y}_1=Y_1+Y_2,\ \overline{Y}_2=Y_1-Y_2$. Then it is clear that $\{\overline{Y}_1,\overline{Y}_2\}$ is a basis of $V$ which satisfies 
\begin{eqnarray*}
&&f\triangleright \overline{Y}_i=-\overline{Y}_i, \ i=1,2,\\
&&e\triangleright \overline{Y}_1=\gamma_2 \overline{Y}_1,  e\triangleright \overline{Y}_2=-\gamma_2 \overline{Y}_2,\\
&&g\triangleright \overline{Y}_1=\beta_2 \overline{Y}_2,  g\triangleright \overline{Y}_2=\beta_2 \overline{Y}_1.
\end{eqnarray*}
The rest of the proof is similar to (1).

(3).  Let $\overline{X}_1=X_1+X_2,\ \overline{X}_2=X_1-X_2$. Then it is clear that $\{\overline{X}_1,\overline{X}_2\}$ is a basis of $U$ which satisfies 
\begin{eqnarray*}
&&e\triangleright \overline{X}_i=-\overline{X}_i, \ i=1,2,\\
&&g\triangleright \overline{X}_1=\gamma_1 \overline{X}_1,  g\triangleright \overline{X}_2=-\gamma_1 \overline{X}_2,\\
&&f\triangleright \overline{X}_1=\beta_1 \overline{X}_2,  f\triangleright \overline{X}_2=\beta_1 \overline{X}_1.
\end{eqnarray*}
Similarly, let $\overline{Z}_1=Z_1+Z_2,\ \overline{Z}_2=Z_1-Z_2$. Then $\{\overline{Z}_1,\overline{Z}_2\}$ is a basis of $W$ and we have 
\begin{eqnarray*}
&&g\triangleright \overline{Z}_i=-\overline{Z}_i, \ i=1,2,\\
&&e\triangleright \overline{Z}_1=\gamma_3 \overline{Z}_1,  f\triangleright \overline{Z}_2=-\gamma_3 \overline{Z}_2,\\
&&f\triangleright \overline{Z}_1=\beta_3 \overline{Z}_2,  e\triangleright \overline{Z}_2=-\beta_3 \overline{Z}_1.
\end{eqnarray*}
The same as (1), if $\gamma_1\gamma_3\neq \pm 1$, then $B(U\oplus W)$ is infinite dimensional.
\end{proof}

If the structure constants of $U,V,W$ satisfy $\beta_2\beta_3= \pm 1$, $\beta_1\gamma_2= \pm 1$ and $\gamma_1\gamma_3= \pm 1$, then one can show that $B(U\oplus V), B(U\oplus W), B(V\oplus W)$ are all finite dimensional, the proof is the same as that of \cite[Proposition 5.1]{HYZ}.  In what follows, we will prove that $B(U\oplus V\oplus W)$ is infinite dimensional. For $\Z^3$-graded homogenous elements $X, Y\in B(U\oplus V\oplus W)$, we denote $\deg(X\otimes Y)=(\deg(X),\ \deg(Y))$.

\begin{proposition}\label{p7.7}
Suppose $\beta_2\beta_3= \pm 1$, $\beta_1\gamma_2= \pm 1$ and $\gamma_1\gamma_3= \pm 1$, then $B(U\oplus V\oplus W)$ is infinite dimensional.
\end{proposition}
\begin{proof}

Let $T(U\oplus V\oplus W)$ be the tensor algebra of $U\oplus V\oplus W$, $\mathcal{I}$ the maximal $\N$-graded Hopf ideal contained in $\bigoplus_{n\geq 2}(U\oplus V\oplus W)^{\otimes \overrightarrow{n}}$. Thus $B(U\oplus V\oplus W)=T(U\oplus V\oplus W)/\mathcal{I}$ by definition.  According to Proposition \ref{p2.9}, $B(U\oplus V\oplus W)$ is $\Z^3$-graded with $\deg(U)=e_1, \deg(V)=e_2, \deg(W)=e_3$, where $\{e_1, e_2, e_3\}$ are free generators of $\Z^3$.
Next we will prove that $B(U\oplus V\oplus W)$ is infinite dimensional. Without loss of generality, we can assume that  $\beta_2\beta_3= -1$, $\beta_1\gamma_2= -1$ and $\gamma_1\gamma_3=-1$ by 
Remark \ref{r5.7}.

The remaining proof will be divided into four steps.

{\bf Step 1.} We will consider the comultiplications of some elements in the spaces $\ad_V(W),$ $\ad_U(V),$ $\ad_U(W) \subset B(U\oplus V\oplus W)$. In $\ad_V(W)$ we have 
\begin{eqnarray*}
\ad_{Y_1}(Z_1)=Y_1Z_1-(f\triangleright Z_1)Y_1=Y_1Z_1-\beta_3 Z_1Y_1,
\end{eqnarray*}
thus
\begin{equation}\label{eq7.11}
\begin{split}
\D(\ad_{Y_1}(Z_1))=&\D(Y_1Z_1-\beta_3 Z_1Y_1)\\
=&(1\otimes Y_1+Y_1\otimes 1)(1\otimes Z_1+ Z_1\otimes 1)\\
&-\beta_3(1\otimes Z_1+Z_1\otimes 1)(1\otimes Y_1+ Y_1\otimes 1)\\
=&1\otimes Y_1Z_1+Y_1Z_1\otimes 1+\beta_3 Z_1\otimes Y_1+Y_1\otimes Z_1\\
&-\beta_3[1\otimes Z_1Y_1+Z_1Y_1\otimes 1+\beta_2 Y_1\otimes Z_1+Z_1\otimes Y_1]\\
=&1\otimes \ad_{Y_1}(Z_1)+ \ad_{Y_1}(Z_1)\otimes 1 +2Y_1\otimes Z_1.
\end{split}
\end{equation}
Similarly, we have 
\begin{eqnarray}
&&\D(\ad_{Y_1}(Z_2))=1\otimes \ad_{Y_1}(Z_2)+ \ad_{Y_1}(Z_2)\otimes 1, \label{eq7.12}\\
&&\D(\ad_{Y_2}(Z_1))=1\otimes \ad_{Y_2}(Z_1)+ \ad_{Y_2}(Z_1)\otimes 1,\label{eq7.13}\\
&&\D(\ad_{Y_2}(Z_2))=1\otimes \ad_{Y_2}(Z_2)+ \ad_{Y_2}(Z_2)\otimes 1 +2Y_2\otimes Z_2. \label{eq7.14}
\end{eqnarray}
The identities \eqref{eq7.12}-\eqref{eq7.13} imply that $\ad_{Y_1}(Z_2)=0, \ad_{Y_2}(Z_1)=0$. 

In $\ad_U(V)$,  we have 

\begin{eqnarray}
&&\D(\ad_{X_1}(Y_1))=1\otimes \ad_{X_1}(Y_1) +\ad_{X_1}(Y_1)\otimes 1+X_1\otimes (Y_1+Y_2), \\
&&\D(\ad_{X_1}(Y_2))=1\otimes \ad_{X_1}(Y_2)+ \ad_{X_1}(Y_2)\otimes 1+X_1\otimes (Y_1+Y_2),\\
&&\D(\ad_{X_2}(Y_1))=1\otimes \ad_{X_2}(Y_1)+\ad_{X_2}(Y_1)\otimes 1+ X_2\otimes(Y_1-Y_2),\\
&&\D(\ad_{X_2}(Y_2))=1\otimes \ad_{X_2}(Y_2)+\ad_{X_2}(Y_2)\otimes 1 +X_2\otimes(Y_2-Y_1).
\end{eqnarray}
It is easy to see that 
\begin{eqnarray}
&&\ad_{X_1}(Y_1)-\ad_{X_1}(Y_2)=0,\\
&&\ad_{X_2}(Y_1)+\ad_{X_2}(Y_2)=0.
\end{eqnarray}

Similarly in $\ad_U(W)$,  we have 

\begin{eqnarray}
&&\D(\ad_{X_1}(Z_1))=1\otimes \ad_{X_1}(Z_1) +\ad_{X_1}(Z_1)\otimes 1+X_1\otimes Z_1+X_2\otimes Z_2,\label{eq7.21} \\
&&\D(\ad_{X_1}(Z_2))=1\otimes \ad_{X_1}(Z_2)+ \ad_{X_1}(Z_2)\otimes 1+X_1\otimes Z_2+X_2\otimes Z_1,\label{eq7.22} \\
&&\D(\ad_{X_2}(Z_1))=1\otimes \ad_{X_2}(Z_1)+\ad_{X_2}(Z_1)\otimes 1+ X_2\otimes Z_1+X_1\otimes Z_2,\label{eq7.23}\\
&&\D(\ad_{X_2}(Z_2))=1\otimes \ad_{X_2}(Z_2)+\ad_{X_2}(Z_2)\otimes 1 + X_1\otimes Z_1+X_2\otimes Z_2.\label{eq7.24}
\end{eqnarray}
By \eqref{eq7.21}-\eqref{eq7.24} we get 
\begin{eqnarray}
&&\ad_{X_1}(Z_1)-\ad_{X_2}(Z_2)=0,\label{eq7.28}\\
&&\ad_{X_1}(Z_2)-\ad_{X_2}(Z_1)=0.\label{eq7.29}
\end{eqnarray}

{\bf Step 2.} We will prove that $\ad_{X_1}(\ad_{Y_1}(Z_1)),$ $\ad_{X_2}(\ad_{Y_2}(Z_2)),$ $\ad_{Y_1}(\ad_{X_1}(Z_2))$ and $\ad_{Y_2}(\ad_{X_2}(Z_1))$ are linear independent in $B(U\oplus V\oplus W)$. Firstly we have 
\begin{equation}\label{eq7.27}
\begin{split}
&\D(\ad_{X_1}(\ad_{Y_1}(Z_1)))\\
&=\D(X_1\ad_{Y_1}(Z_1)-e\triangleright (\ad_{Y_1}(Z_1))X_1)\\
&=\D(X_1\ad_{Y_1}(Z_1)+\gamma_2\gamma_3 \ad_{Y_2}(Z_2) X_1)\\
&=(1\otimes X_1+X_1\otimes 1) (1\otimes \ad_{Y_1}(Z_1)+ \ad_{Y_1}(Z_1)\otimes 1 +2Y_1\otimes Z_1)\\
&\ \ \ +\gamma_2\gamma_3 (1\otimes \ad_{Y_2}(Z_2)+ \ad_{Y_2}(Z_2)\otimes 1 +2Y_2\otimes Z_2)(1\otimes X_1+X_1\otimes 1)\\
&=1\otimes \ad_{X_1}(\ad_{Y_1}(Z_1)) +\ad_{X_1}(\ad_{Y_1}(Z_1))\otimes 1+ X_1\otimes \ad_{Y_1}(Z_1)\\
&\ \ \ +X_2\otimes  \ad_{Y_2}(Z_2)-2\gamma_2Y_2\otimes \ad_{X_1}(Z_1) -2X_1Y_1\otimes Z_1-2\gamma_2Y_2X_2\otimes Z_2.
\end{split}
\end{equation}
Here the third identity follows from \eqref{eq7.11} and \eqref{eq7.14}. Similarly we have 

\begin{equation}\label{eq7.28}
\begin{split}
&\D(\ad_{X_2}(\ad_{Y_2}(Z_2)))\\
&=\D(X_2\ad_{Y_2}(Z_2)-e\triangleright (\ad_{Y_2}(Z_2))X_2)\\
&=\D(X_2\ad_{Y_2}(Z_2)+\gamma_2\gamma_3 \ad_{Y_1}(Z_1) X_2)\\
&=(1\otimes X_2+X_2\otimes 1)\Big (1\otimes \ad_{Y_2}(Z_2)+ \ad_{Y_2}(Z_2)\otimes 1 +2Y_2\otimes Z_2\Big)\\
&\ \ \ +\gamma_2\gamma_3 \Big(1\otimes \ad_{Y_1}(Z_1)+ \ad_{Y_1}(Z_1)\otimes 1 +2Y_1\otimes Z_1\Big)(1\otimes X_2+X_2 \otimes 1)\\
&=1\otimes \ad_{X_2}(\ad_{Y_2}(Z_2)) + \ad_{X_2}(\ad_{Y_2}(Z_2)) \otimes 1-X_1\otimes \ad_{Y_1}(Z_1)\\
&\ \ \ +X_2\otimes  \ad_{Y_2}(Z_2)-2\gamma_2Y_1\otimes \ad_{X_2}(Z_2)-2X_2Y_2\otimes Z_2-2\gamma_2Y_1X_1\otimes Z_1,
\end{split}
\end{equation}

\begin{equation}\label{eq7.29}
\begin{split}
&\D(\ad_{Y_1}(\ad_{X_1}(Z_2)))\\
&=\D(Y_1\ad_{X_1}(Z_2)-f\triangleright (\ad_{X_1}(Z_2))Y_1)\\
&=\D(Y_1\ad_{X_1}(Z_2)-\beta_1\beta_3 \ad_{X_1}(Z_2) Y_1)\\
&=(1\otimes Y_1+Y_1\otimes 1) \\
&\ \ \ \ \ \ \ \ \ \ \ \ \ \ \ \times\Big(1\otimes \ad_{X_1}(Z_2)+ \ad_{X_1}(Z_2)\otimes 1+X_1\otimes Z_2+X_2\otimes Z_1\Big)\\
&\ \ \ -\beta_1\beta_3 \Big(1\otimes \ad_{X_1}(Z_2)+ \ad_{X_1}(Z_2)\otimes 1+X_1\otimes Z_2+X_2\otimes Z_1\Big)\\
&\ \ \ \ \ \ \ \ \ \ \ \ \ \ \ \ \ \ \ \ \ \ \ \ \ \ \ \ \ \ \ \ \ \ \ \ \ \ \ \ \ \ \ \ \ \ \ \ \ \ \ \ \ \ \ \ \ \ \ \ \ \ \ \ \times (1\otimes Y_1+Y_1\otimes 1)\\
&=1\otimes \ad_{Y_1}(\ad_{X_1}(Z_2)) +\ad_{Y_1}(\ad_{X_1}(Z_2))\otimes 1+ (Y_1+Y_2)\otimes \ad_{X_1}(Z_2)\\
&\ \ \ +\beta_1X_2\otimes \ad_{Y_1}(Z_1) +(Y_1X_2-\beta_1X_2Y_1)\otimes Z_1+\ad_{Y_1}(X_1)\otimes Z_2
\end{split}
\end{equation}
and 
\begin{equation}\label{eq7.30}
\begin{split}
&\D(\ad_{Y_2}(\ad_{X_2}(Z_1)))\\
&=\D(Y_2\ad_{X_2}(Z_1)-f\triangleright (\ad_{X_2}(Z_1))Y_2)\\
&=\D(Y_2\ad_{X_2}(Z_1)-\beta_1\beta_3 \ad_{X_2}(Z_1) Y_2)\\
&=(1\otimes Y_2+Y_2\otimes 1) \\
&\ \ \ \ \ \ \ \ \ \ \ \ \ \ \ \times\Big(1\otimes \ad_{X_2}(Z_1)+\ad_{X_2}(Z_1)\otimes 1+ X_2\otimes Z_1+X_1\otimes Z_2\Big)\\
&\ \ \ -\beta_1\beta_3 \Big(1\otimes \ad_{X_2}(Z_1)+\ad_{X_2}(Z_1)\otimes 1+ X_2\otimes Z_1+X_1\otimes Z_2\Big)\\
&\ \ \ \ \ \ \ \ \ \ \ \ \ \ \ \ \ \ \ \ \ \ \ \ \ \ \ \ \ \ \ \ \ \ \ \ \ \ \ \ \ \ \ \ \ \ \ \ \ \ \ \ \ \ \ \ \ \ \ \ \ \ \ \ \times (1\otimes Y_2+Y_2\otimes 1)\\
&=1\otimes \ad_{Y_2}(\ad_{X_2}(Z_1)) +\ad_{Y_2}(\ad_{X_2}(Z_1))\otimes 1+ (Y_2-Y_1)\otimes \ad_{X_2}(Z_1)\\
&\ \ \ -\beta_1X_1\otimes \ad_{Y_2}(Z_2) +(Y_2X_1+\beta_1X_1Y_2)\otimes Z_2+\ad_{Y_2}(X_2)\otimes Z_1.
\end{split}
\end{equation}

Now let $k_1,k_2,k_3,k_4\in \k$ such that $$k_1 \ad_{X_1}(\ad_{Y_1}(Z_1))+k_2\ad_{X_2}(\ad_{Y_2}(Z_2))+k_3 \ad_{Y_1}(\ad_{X_1}(Z_2))+k_4 \ad_{Y_2}(\ad_{X_2}(Z_1))=0.$$
By \eqref{eq7.27}-\eqref{eq7.30}, the homogenous term in the comultiplication of the left hand side of the above identity with degree $(e_2, e_1+e_3)$  is
\begin{equation*}
\begin{split}
&-2k_1\gamma_2Y_2\otimes \ad_{X_1}(Z_1)-2k_2\gamma_2Y_1\otimes \ad_{X_2}(Z_2)\\
&\ \ +k_3 (Y_1+Y_2)\otimes \ad_{X_1}(Z_2)+k_4(Y_2-Y_1)\otimes \ad_{X_2}(Z_1)\\
&=-2\gamma_2(k_1Y_2+k_2Y_1)\otimes \ad_{X_1}(Z_1)+ ((k_3-k_4)Y_1+(k_3+k_4)Y_2)\otimes  \ad_{X_1}(Z_2)\\
\end{split}
\end{equation*}
On the other hand, this element must be zero since comultiplication keep $Z^3$-degrees.  This implies $k_1=k_2=k_3=k_4=0$ because $\ad_{X_1}(Z_1)$ and  $\ad_{X_1}(Z_2)$ are linear independent by \eqref{eq7.21}-\eqref{eq7.22} and 
$\{Y_1,Y_2\}$ is a basis of $V$.

{\bf Step 3.} Let 
\begin{eqnarray*}
&&E=\ad_{X_1}(\ad_{Y_1}(Z_1)),\ \ \ \ \  F=\ad_{X_2}(\ad_{Y_2}(Z_2)),\\
&&M=\ad_{Y_1}(\ad_{X_1}(Z_2)),\ \ \ \ N=\ad_{Y_2}(\ad_{X_2}(Z_1)).
\end{eqnarray*}
 We will prove that $\ad_{E}(M)\neq 0 $ and can not be linear spanned by $M^2, MN,NM, N^2$. 
Firstly we have 
\begin{equation}
\begin{split}
&(efg) \triangleright E\\
&=\widetilde{\Phi}_{efg}(e,f)\widetilde{\Phi}_{efg}(ef,g) e\triangleright\{f\triangleright[g\triangleright (\ad_{X_1}(\ad_{Y_1}(Z_1))]\}\\
&=-e\triangleright\{f\triangleright[g\triangleright (\ad_{X_1}(\ad_{Y_1}(Z_1))]\}\\
&=-\beta_2\gamma_1 e\triangleright [f\triangleright (\ad_{X_2}(\ad_{Y_1}(Z_1))]\\
&=\beta_2\gamma_1 \beta_1\beta_3 e\triangleright (\ad_{X_2}(\ad_{Y_1}(Z_1))\\
&=\beta_2\gamma_1 \beta_1\beta_3 \gamma_2\gamma_3 (\ad_{X_2}(\ad_{Y_2}(Z_2))\\
&=-F.
\end{split}
\end{equation}
Similarly, one can show that $(efg) \triangleright F=E, (efg) \triangleright M=-N, (efg) \triangleright N=M.$ By \eqref{eq7.27}-\eqref{eq7.30}, the comultiplications of $E,F,M$ and $N$ can be written as  the forms of 
\begin{eqnarray}
&&\D(E)=E\otimes 1+1\otimes E+\sum \overline{E}_1\otimes \overline{E}_2,\label{eq7.32}\\
&&\D(F)=F\otimes 1+1\otimes F+\sum \overline{F}_1\otimes \overline{F}_2,\label{eq7.33}\\
&&\D(M)=M\otimes 1+1\otimes M+\sum \overline{M}_1\otimes \overline{M}_2,\label{eq7.34}\\
&&\D(N)=N\otimes 1+1\otimes N+\sum \overline{N}_1\otimes \overline{N}_2,\label{eq7.35}
\end{eqnarray}
Where $\deg(\overline{E}_2), \deg(\overline{F}_2), \deg(\overline{M}_2),\deg(\overline{N}_2)\in \{e_1+e_3, e_2+e_3, e_3\}$. In the following, let $\mathcal{Z}$ be the subset of $\Z^3$ defined by 
\begin{equation}
\mathcal{Z}=\{k_1e_1+k_2e_2+k_3e_3|\ k_1,k_2\leq k_3, \min\{k_1,k_2\}<k_3\}.
\end{equation}
Here $\min\{k_1,k_2\}$ means the smaller number of $\{k_1,k_2\}$. It is clear that $\deg(\overline{E}_2)$, $\deg(\overline{F}_2)$, $\deg(\overline{M}_2)$, $\deg(\overline{N}_2)\in \mathcal{Z}$. Moreover, for all homogenous elements $X,Y,Z \in B(U\oplus V\oplus W)$ such that $\deg(X),\deg(Y)\in \mathcal{Z}$, $\deg(Z)=k(e_1+e_2+e_3)$, we have $\deg(XY)\in \mathcal{Z},\  \deg(XZ)=\deg(ZX)\in \mathcal{Z}$.

 By \eqref{eq7.32}-\eqref{eq7.35}, we have 
\begin{equation}\label{eq7.37}
\begin{split}
&\D(\ad_{E}(M))\\
&=\D(EM-efg\triangleright ME)=\D(EM+NE)\\
&=\Big(E\otimes 1+1\otimes E+\sum \overline{E}_1\otimes \overline{E}_2\Big)\Big(M\otimes 1+1\otimes M+\sum \overline{M}_1\otimes \overline{M}_2\Big)\\
&\ \ +\Big(N\otimes 1+1\otimes N+\sum \overline{N}_1\otimes \overline{N}_2\Big)\Big(E\otimes 1+1\otimes E+\sum \overline{E}_1\otimes \overline{E}_2\Big)\\
&=\ad_{E}(M)\otimes 1+1\otimes \ad_{E}(M)+E\otimes M+F\otimes N+\sum T_1\otimes T_2,
\end{split}
\end{equation}
where 
\begin{eqnarray*}
\sum T_1\otimes T_2&=&(E\otimes 1+1\otimes E)(\sum \overline{M}_1\otimes \overline{M}_2)+(\sum \overline{E}_1\otimes \overline{E}_2)(M\otimes 1+1\otimes M)\\
&&+(N\otimes 1+1\otimes N)(\sum \overline{E}_1\otimes \overline{E}_2)+(\sum \overline{N}_1\otimes \overline{N}_2)(E\otimes 1+1\otimes E)\\
&&+(\sum \overline{E}_1\otimes \overline{E}_2)(\sum \overline{M}_1\otimes \overline{M}_2)+(\sum \overline{N}_1\otimes \overline{N}_2)(\sum \overline{E}_1\otimes \overline{E}_2).
\end{eqnarray*}
It is easy to see that $\deg(T_2)\in \mathcal{Z}$, so $\D(\ad_{E}(M))\neq 0$ since $E\otimes M+F\otimes N\neq 0$.
Next consider the comultiplications of $M^2, MN,NM, N^2$. Since  
\begin{eqnarray*}
\D(M^2)&=&(M\otimes 1+1\otimes M+\sum \overline{M}_1\otimes \overline{M}_2)(M\otimes 1+1\otimes M+\sum \overline{M}_1\otimes \overline{M}_2)\\
&=&M^2\otimes1 +1\otimes M^2+(M-N)\otimes M+(M\otimes 1+1\otimes M)(\sum \overline{M}_1\otimes \overline{M}_2)\\
&&+ (\sum \overline{M}_1\otimes \overline{M}_2)(M\otimes 1+1\otimes M)+(\sum \overline{M}_1\otimes \overline{M}_2)(\sum \overline{M}_1\otimes \overline{M}_2).
\end{eqnarray*}
The term in $\D(M^2)$ with degree $(e_1+e_2+e_3, e_1+e_2+e_3)$ is $(M-N)\otimes M$. Similarly, the terms in $\D(MN),\D(NM)$ and $\D(N^2)$ with degree $(e_1+e_2+e_3, e_1+e_2+e_3)$ are $M\otimes (N+M)$, $N\otimes (M-N)$ and $(M+N)\otimes N$
respectively. Now suppose that there exist $k_1,k_2,k_3,k_4$ such that 
\begin{equation}\label{eq7.38}
\ad_{E}(M)=k_1M^2+k_2MN+k_3NM+k_4N^2.
\end{equation} Considering the terms with degree $(e_1+e_2+e_3, e_1+e_2+e_3)$ in the comultiplications of the both sides of \eqref{eq7.38}, we obtain 
\begin{eqnarray*}
&&E\otimes M+F\otimes N\\
&=&k_1((M-N)\otimes M)+k_2M\otimes (N+M)+k_3N\otimes (M-N)+k_4(M+N)\otimes N\\
&=&M\otimes [(k_1+k_2)M+(k_2+k_4)N]+N\otimes[(k_3-k_1)M+(k_4-k_3)N].
\end{eqnarray*}
But this is impossible since $E,F,M,N$ are linear independent.
We have proved that $\ad_{E}(M)$ can not be linearly spanned by $M^2, MN,NM, N^2$.

{\bf Step 4.} We will prove that $(\ad_{E}(M))^n\neq 0$ for all $n\in \mathbbm{N}$ inductively, and this clearly implies that $B(U\otimes V\otimes W)$ is infinite dimensional.  In Step 3, we have show that $\ad_{E}(M)\neq 0$. Next suppose $(\ad_{E}(M))^{n-1}\neq 0$.

By \eqref{eq7.37},  $\D(\ad_{E}(M))=\ad_{E}(M)\otimes 1+1\otimes \ad_{E}(M)+E\otimes M+F\otimes N+\sum T_1\otimes T_2$, where $\deg(T_2)\in \mathcal{Z}$.  So the terms in $\D((\ad_{E}(M))^n)$ with degree of the form $$(k(e_1+e_2+e_3),l(e_1+e_2+e_3)), \ k,l\in \N$$  must be contained in
$$[\ad_{E}(M)\otimes 1+1\otimes \ad_{E}(M)+E\otimes M+F\otimes N]^n.$$
So the term in $\D((\ad_{E}(M))^n)$ with degree $((2n-2)(e_1+e_2+e_3), 2(e_1+e_2+e_3))$ is $$n(\ad_{E}(M))^{n-1}\otimes \ad_{E}(M)+S_1\otimes M^2+S_2\otimes MN+S_3\otimes NM+S_4\otimes N^2,$$ where $S_1,S_2,S_2,S_4$ are certain elements with degree $(2n-2)(e_1+e_2+e_3)$. By hypothesis, we have $n(\ad_{E}(M))^{n-1}\otimes \ad_{E}(M)\neq 0$. On the other hand, since $\ad_{E}(M)$ can not be spanned by $M^2, MN,NM, N^2$, we obtain $$n(\ad_{E}(M))^{n-1}\otimes \ad_{E}(M)+S_1\otimes M^2+S_2\otimes MN+S_3\otimes NM+S_4\otimes N^2\neq 0.$$ This implies $\D((\ad_{E}(M))^n)\neq 0$ and hence $(\ad_{E}(M))^n\neq 0$.
\end{proof}

It is clear that Proposition \ref{p7.1} follows from Propositions \ref{p7.6} and \ref{p7.7}. 

\subsection{A proof of Theorem \ref{t3.9}}

In this subsection, let $G=\langle g_1\rangle \times \langle g_2\rangle \times g_3\rangle$, $\Phi$ be a nonabelian $3$-cocycle on $G$, $V_1, V_2, V_3\in  {_{\k G}^{\k G} \mathcal{YD}^\Phi}$ 
be simple twisted Yetter-Drinfeld modules of type (C1) such that $\dim(V_i)=2$,  $\deg(V_i)=g_i$ for $1\leq i\leq 3$. In what follows, we denote $m_i=|g_i|$, $i=1,2,3$. By \eqref{eq2.4}, we have 
\begin{equation}\label{eq7.31}
\begin{split}
&\Phi(g_{1}^{i_{1}}g_{1}^{i_{2}}g_{1}^{i_{3}},g_{1}^{j_{1}}g_{1}^{j_{2}}g_{1}^{j_{3}},g_{1}^{k_{1}}g_{1}^{k_{2}}g_{1}^{k_{3}})\\
=& \prod_{l=1}^{3}\zeta_{m_l}^{c_{l}i_{l}[\frac{j_{l}+k_{l}}{m_{l}}]}
\prod_{1\leq s<t\leq 3}\zeta_{m_{t}}^{c_{st}i_{t}[\frac{j_{s}+k_{s}}{m_{s}}]}
\times \zeta_{(m_{1},m_{2},m_{3})}^{c_{123}i_{1}j_{2}k_{3}},
\end{split}
\end{equation}
Where $0\leq c_{l}< m_l$ for $1\leq l\leq 3$, $0\leq c_{st}<m_t$ for $1\leq s<t\leq 3$, $0\leq c_{123}<(m_{1},m_{2},m_{3})$. Furthermore , we have the following lemma.
\begin{lemma}\label{l7.8}
With the notations above, we have $c_{123}=\frac{(m_{1},m_{2},m_{3})}{2}$, that is 
\begin{equation}
\zeta_{(m_{1},m_{2},m_{3})}^{c_{123}k_{1}j_{2}i_{3}}=(-1)^{k_{1}j_{2}i_{3}}.
\end{equation}
\end{lemma}
\begin{proof}
By Lemma \ref{l3.5}, $\dim(V_1)=\Big|\frac{\widetilde{\Phi}_{g_1}(g_2,g_3)}{\widetilde{\Phi}_{g_1}(g_3,g_2)}\Big|$, the order of $\frac{\widetilde{\Phi}_{g_1}(g_2,g_3)}{\widetilde{\Phi}_{g_1}(g_3,g_2)}.$ Since $\dim(V_i)=2$ for $1\leq i\leq 3$, we have 
\begin{equation}\label{eq7.33}
\frac{\widetilde{\Phi}_{g_1}(g_2,g_3)}{\widetilde{\Phi}_{g_1}(g_3,g_2)}=-1.
\end{equation}
By \eqref{eq7.31} and \eqref{eq7.33}, we obtain $$\frac{\widetilde{\Phi}_{g_1}(g_2,g_3)}{\widetilde{\Phi}_{g_1}(g_3,g_2)}=\Phi(g_1,g_2,g_3)=\zeta_{(m_{1},m_{2},m_{3})}^{c_{123}}=-1.$$
\end{proof}

{\bf Proof of Theorem  \ref{t3.9}}.
Let $\Psi$ and $\Gamma$ be the $3$-cocycles of $G$ given by 
\begin{eqnarray}
&\Psi(g_{1}^{i_{1}}g_{1}^{i_{2}}g_{1}^{i_{3}},g_{1}^{j_{1}}g_{1}^{j_{2}}g_{1}^{j_{3}},g_{1}^{k_{1}}g_{1}^{k_{2}}g_{1}^{k_{3}})=(-1)^{i_1j_2k_3}, \\
&\Gamma(g_{1}^{i_{1}}g_{1}^{i_{2}}g_{1}^{i_{3}},g_{1}^{j_{1}}g_{1}^{j_{2}}g_{1}^{j_{3}},g_{1}^{k_{1}}g_{1}^{k_{2}}g_{1}^{k_{3}})=\prod_{l=1}^{3}\zeta_{m_l}^{c_{l}i_{l}[\frac{j_{l}+k_{l}}{m_{l}}]}
\prod_{1\leq s<t\leq 3}\zeta_{m_{t}}^{c_{st}i_{t}[\frac{j_{s}+k_{s}}{m_{s}}]}.
\end{eqnarray}
By \eqref{eq7.31} and Lemma \ref{l7.8}, we have $\Phi=\Gamma \times \Psi$. By \eqref{eq2.13}, $\Gamma$ is an abelian $3$-cocycle of $G$. Let $\widehat{G}=\langle \mathbbm{g}_1\rangle \times \langle \mathbbm{g}_2\rangle \times \langle \mathbbm{g}_3\rangle$ such that $|\mathbbm{g}_i|=m_i^2, 1\leq i\leq 3$, and $\pi\colon \widehat{G}\To G$ be the epimorphism determined by 
\begin{equation}
\pi(\mathbbm{g}_i)=g_i, \ 1\leq i\leq 3.
\end{equation}
Let $\iota\colon G\To \widehat{G}$ be the section of $\pi$ given by 
\begin{eqnarray}
\iota(g_i^l)=\mathbbm{g}_i^l,\ \ 0\leq l<m_i.
\end{eqnarray}
Then we have an object $\widetilde{V}=\widetilde{V}_1\oplus \widetilde{V}_2\oplus \widetilde{V}_3 \in {^{\k \widehat{G}}_{\k \widehat{G}} \mathcal{YD}^{\pi^*\Phi}}$ defined by \eqref{eq3.1}-\eqref{eq3.2}, and $B(\widetilde{V})\cong B(V)$ by Lemma \ref{l2.8}. 
By Proposition \ref{p2.6}, $\pi^*\Gamma$ is a $3$-coboundary of $\widehat{G}$. Let $J$ be the $2$-cochain of $\widehat{G}$ such that $\partial J=\pi^*\Gamma$. So we have 
$$\partial(J^{-1})\cdot \pi^*\Phi=(\pi^*\Gamma)^{-1}\cdot \pi^*(\Psi \Gamma)=\pi^*\Psi.$$
By lemma \ref{l2.6}, we have 
$$B(\widetilde{V})^{J^{-1}}\cong B(\widetilde{V}^{J^{-1}})\in {^{\k \widehat{G}}_{\k \widehat{G}} \mathcal{YD}^{\partial J \cdot \pi^*\Phi}}= {^{\k \widehat{G}}_{\k \widehat{G}} \mathcal{YD}^{\pi^*\Psi}}.$$
Note that $\deg(\widetilde{V}_i^{J^{-1}})=\deg(\widetilde{V}_i)=\mathbbm{g}_i$ for $1\leq i\leq 3$, and 
\begin{equation*}
\begin{split}
&\pi^*\Psi(\mathbbm{g}_{1}^{i_{1}}\mathbbm{g}_{1}^{i_{2}}\mathbbm{g}_{1}^{i_{3}},\mathbbm{g}_{1}^{j_{1}}\mathbbm{g}_{1}^{j_{2}}\mathbbm{g}_{1}^{j_{3}},\mathbbm{g}_{1}^{k_{1}}\mathbbm{g}_{1}^{k_{2}}\mathbbm{g}_{1}^{k_{3}})\\
&=\Psi(\pi(\mathbbm{g}_{1}^{i_{1}}\mathbbm{g}_{1}^{i_{2}}\mathbbm{g}_{1}^{i_{3}}),\pi(\mathbbm{g}_{1}^{j_{1}}\mathbbm{g}_{1}^{j_{2}}\mathbbm{g}_{1}^{j_{3}}),\pi(\mathbbm{g}_{1}^{k_{1}}\mathbbm{g}_{1}^{k_{2}}\mathbbm{g}_{1}^{k_{3}}))\\
&=\Psi(g_{1}^{i_{1}}g_{1}^{i_{2}}g_{1}^{i_{3}},g_{1}^{j_{1}}g_{1}^{j_{2}}g_{1}^{j_{3}},g_{1}^{k_{1}}g_{1}^{k_{2}}g_{1}^{k_{3}})\\
&=(-1)^{k_1j_2i_3}
\end{split}
\end{equation*}
for all $ 0\leq i_l,j_l,k_l<m_l^2, 1\leq l\leq 3$. By Proposition \ref{p7.1}, $B(\widetilde{V})^{J^{-1}}$ is infinite dimensional. So $B(\widetilde{V})$ and $B(V)$ are also infinite dimensional.
 \hfill $\Box$

\section{Finite quasi-quantum groups over abelian groups}

In this section, we will give a classification of finite-dimensional coradically graded pointed coquasi-Hopf algebras over finite abelian groups.

Let $M=\oplus_{i\geq 0}M_i$ be a coradically graded pointed coquasi-Hopf algebra over a finite abelian group $G$. Then $M_0=(\k G,\Phi)$ for a $3$-cocycle $\Phi$ on $G$. Let $R=\oplus_{i\geq 0}R[i]$ be the coinvariant subalgebra of $M$.
With these  notations, we have 
\begin{theorem}\label{t9.1}
Assume that $M$ is finite dimensional. Then coinvariant subalgebra $R$ of $M$ is a Nichols algebra in  ${_{\k G}^{\k G} \mathcal{YD}^\Phi}$.
\end{theorem}
\begin{proof}
First we have $G_{R[1]}=G_{R}$ for $R$ is coradically graded. Since $R$  is finite dimensional, we have $B(R[1])$ is also finite dimensional since $B(R[1])$ is a subquotient of $R$. This implies that $\Phi_{G_{R[1]}}$ is an abelian $3$-cocycle on $G_{R[1]}$
by Corollary \ref{c3.15}. So we have $R\cong B(R[1])$ by \cite[Proposition 5.1]{HLYY2}.
\end{proof}

Now we can give a classification of finite-dimensional coradically graded pointed coquasi-Hopf algebras over finite abelian groups.

\begin{theorem}\label{t5.2}
\begin{itemize}
\item[(1).] Let $V\in {_{\k G}^{\k G} \mathcal{YD}^\Phi}$ be a Yetter-Drinfeld module of finite type. Then $ B(V)\#\k G$ is a finite-dimensional pointed coquasi-Hopf algebra.
\item[(2).] Let $M$ be a finite-dimensional coradically graded pointed coquasi-Hopf algebra over a finite abelian group, $M_0=(G,\Phi)$. Then we have $M\cong B(V)\#\k G$ for a Yetter-Drinfeld module of finite type 
$V\in {_{\k G}^{\k G} \mathcal{YD}^\Phi}$.
\end{itemize}
\end{theorem}
\begin{proof}
(1). It follows from Theorem \ref{t8.8} that $B(V)$ is finite dimensional, thus $B(V)\#\k G$ is finite dimensional. 

(2). Let $R$ be the coinvariant subalgebra of $M$. Then $M= R\#\k G$. Let $V=R[1]$. Then by Theorem \ref{t9.1},  $R=B(V)$ is a Nichols algebra in  ${_{\k G}^{\k G} \mathcal{YD}^\Phi}$. So $M=B(V)\#\k G$ for the Yetter-Drinfeld module $V$, and $V$ is of finite type since $M$ and hence $B(V)$ is finite dimensional.
\end{proof}

Finally, we will consider the generation problem of pointed finite tensor categories. We partially prove the following conjecture due to Etingof, Gelaki, Nikshych and Ostrik.
\begin{conjecture}
Every pointed finite tensor categories over a field of characteristic zero is tensor generated by objects of length two.
\end{conjecture}

In fact, this conjecture can be viewed as a generalization of Andruskiewitsch-Schneider conjecture. Let $\mathcal{C}$ be a pointed finite tensor categories. Then it is well known that $\mathcal{C}\cong \comod(M)$ for a finite-dimensional pointed coquasi-Hopf algebra $M$, see \cite{EO} for details. In \cite{HYZ}, we prove the following proposition.

\begin{proposition}\cite[Proposition 4.10]{HYZ}\label{p9.4}
Let $\mathcal{C}$ be a pointed finite tensor category, and $M$ a finite-dimensional pointed coquasi-Hopf algebra such that $\mathcal{C}\cong \comod(M)$. Then $\mathcal{C}$ is tensor generated by objects of length two if and only if $M$ is generated by 
group-like elements and skew-primitive elements.
\end{proposition}

With the help of Theorem \ref{t5.2} and Proposition \ref{p9.4}, we can prove the following theorem.

\begin{theorem}\label{t5.5}
Let $\mathcal{C}$ be a pointed finite tensor category over a field of characteristic zero such that $G(\mathcal{C})$ is an abelian group. Then $\mathcal{C}$ is tensor generated by objects of length two.
\end{theorem}

\begin{proof}
Let $M$ be a finite-dimensional pointed coquasi-Hopf algebra such that $\mathcal{C}\cong \comod(M)$. By Proposition \ref{p9.4}, we only need to show that $M$ is generated by group-like elements and skew-primitive elements. Let $k$ be the based field of $M$,
$K$ the algebraically closure of $k$. Let $\widetilde{M}=M\otimes_{k}K$. Then $\widetilde{M}$ is also a pointed coquasi-Hopf algebra with the structure induced from that of $M$, and it is obvious that $M$ is generated by group-like elements and skew-primitive elements if and only if $\widetilde{M}$ is generated by group-like elements and skew-primitive elements. On the other hand, we have $\gr(\widetilde{M})\cong B(V)\#\k G$ by Theorem \ref{t5.2}, where $G=G(\mathcal{C})$ and $V$ is a Yetter-Drinfeld module of finite type in ${_{\k G}^{\k G} \mathcal{YD}^\Phi}$ for some $3$-cocyce $\Phi$ on $G$. Thus $\gr(\widetilde{M})$, and hence $\widetilde{M}$ are generated by group-like elements $G$ and skew-primitive elements $V$.
Therefore $M$ is generated by group-like elements and skew-primitive elements.
\end{proof}

 \section*{acknowledgments}
This work is supported by the National Natural Science Foundation of China (Nos. 12271243, 12371037, 12131015, 12161141001 and 12371042), the Fundamental Research Funds for the Central Universities (No. SWU-XDJH202305), the Innovation Program for Quantum Science and Technology (No. 2021ZD0302902), and the Natural Science Foundation of Chongqing (No. cstc2021 jcyj-msxmX0714). The first author is also supported by Fuzhou-Xiamen-Quanzhou National Independent Innovation Demonstration Zone Collaborative Innovation Platform under Grant 2022FX5.

\end{document}